\documentclass[12pt,reqno]{amsart}
\usepackage{amsmath,amssymb,verbatim,comment,color}
\usepackage[pagebackref,pdftex,hyperindex]{hyperref}

\usepackage[margin=3cm]{geometry}
\usepackage{color}

\makeatletter
\@namedef{subjclassname@2020}{\textup{2020} Mathematics Subject Classification}
\makeatother

\newcommand{\C}{\mathbb{C}}

\newcommand{\N}{\mathbb{N}}

\newcommand{\R}{\mathbb{R}}

\newcommand{\cJ}{\mathcal{J}}

\newcommand{\cP}{\mathcal{P}}
\newcommand{\cQ}{\mathcal{Q}}

\newcommand{\cT}{\mathcal{T}}
\newcommand{\cU}{\mathcal{U}}
\newcommand{\cV}{\mathcal{V}}

\newcommand{\cX}{\mathcal{X}}
\newcommand{\cY}{\mathcal{Y}}

\renewcommand{\a}{\alpha}
\renewcommand{\b}{\beta}
\newcommand{\g}{\gamma}
\renewcommand{\d}{\delta}
\newcommand{\e}{\varepsilon}
\newcommand{\s}{\sigma}

\renewcommand{\r}{\rho}

\renewcommand{\i}{\sqrt{-1}}
\newcommand{\p}{\partial}

\newcommand{\dd}{\sqrt{-1}\partial \bar{\partial}}
\newcommand{\ddt}{\frac{d}{dt}}
\newcommand{\hpi}{\frac{\pi}{2}}

\newcommand{\cf}{{\rm cf.\ }}
\newcommand{\eg}{{\rm e.g.\ }}
\newcommand{\ie}{{\rm i.e.\ }}

\renewcommand{\Re}{\mathrm{Re}}
\renewcommand{\Im}{\mathrm{Im}}

\DeclareMathOperator{\arccot}{arccot}

\DeclareMathOperator{\diag}{diag}

\DeclareMathOperator{\Herm}{Herm}

\DeclareMathOperator{\Mat}{Mat}
\DeclareMathOperator{\PSH}{PSH}

\DeclareMathOperator{\sing}{sing}

\DeclareMathOperator{\Stab}{Stab}

\DeclareMathOperator{\Vol}{Vol}

\newcommand{\ve}{\varepsilon}

\renewcommand{\leq}{\leqslant}
\renewcommand{\geq}{\geqslant}

\renewcommand{\hat}{\widehat}
\renewcommand{\tilde}{\widetilde}

\numberwithin{equation}{section}       

\newtheorem{prop} {Proposition} [section]
\newtheorem{thm}[prop] {Theorem}
\newtheorem{dfn}[prop] {Definition}
\newtheorem{lem}[prop] {Lemma}
\newtheorem{cla}[prop] {Claim}
\newtheorem{cor}[prop] {Corollary}
\newtheorem{conj}[prop] {Conjecture}

\theoremstyle{remark}
\newtheorem*{ackn}{\bf{Acknowledgment}}
\newtheorem{rk}[prop]{Remark}

\title[Supercritical deformed Hermitian--Yang--Mills equation]{A Nakai--Moishezon type criterion for supercritical deformed Hermitian--Yang--Mills equation}
\date{\today}

\author[J. Chu]{Jianchun Chu}
\address[Jianchun Chu]{Department of Mathematics\\ Northwestern University\\ 2033 Sheridan Road\\ Evanston, IL 60208}
\email{jianchun@math.northwestern.edu}

 \author[M.-C. Lee]{Man-Chun Lee}
\address[Man-Chun Lee]{Mathematics Institute, Zeeman Building,
University of Warwick, Coventry CV4 7AL; Department of Mathematics, Northwestern University, 2033 Sheridan Road, Evanston, IL 60208}
\email{Man.C.Lee@warwick.ac.uk, mclee@math.northwestern.edu}

\author[R. Takahashi]{Ryosuke Takahashi}
\address[Ryosuke Takahashi]{Faculty of Mathematics\\
 Kyushu University\\
744\\
Motooka\\
Nishi-ku\\
Fukuoka\\
819-0395\\
 JAPAN}
\email{rtakahashi@math.kyushu-u.ac.jp}


\begin{document}
\maketitle
\begin{abstract}
The deformed Hermitian--Yang--Mills equation is a complex Hessian equation on compact K\"ahler manifolds that corresponds to the special Lagrangian equation in the context of the Strominger--Yau--Zaslow mirror symmetry \cite{SYZ96}. Recently, Chen \cite{Che21} proved that the existence of the solution is equivalent to a uniform stability condition in terms of holomorphic intersection numbers along test families. In this paper, we establish an analogous stability result not involving a uniform constant in accordance with a recent work on the $J$-equation by Song \cite{Son20}, which makes further progress toward Collins--Jacob--Yau's original conjecture \cite{CJY15} in the supercritical phase case. In particular, we confirm this conjecture for projective manifolds in the supercritical phase case.
\end{abstract}
\tableofcontents

\section{Introduction}
Let $X$ be an $n$-dimensional compact complex manifold with a real $(1,1)$-cohomology class $\a$ and a K\"ahler class $\b$. For a given K\"ahler form $\chi \in \b$, the {\it deformed Hermitian--Yang--Mills} (dHYM) equation is defined by
\begin{equation} \label{dHYM equation}
\Im \big(e^{-\i \vartheta_0}(\chi+\i \omega)^n \big)=0,
\end{equation}
where $\omega \in \a$ is the desired real $(1,1)$-form and $\vartheta_0$ is a constant. Integrating \eqref{dHYM equation}, one can easily observe that the constant $\vartheta_0$ is uniquely determined (mod. $2 \pi$) by a cohomological condition
\[
\bigg( \int_X (\chi+\i \omega)^n \bigg) e^{-\i \vartheta_0} \in \R_{>0}.
\]
The dHYM equation first appeared in the physics literature \cite{MMMS00} and \cite{LYZ01} in the mathematical side. The dHYM equation corresponds to the special Lagrangian equation in the setting of the Strominger--Yau--Zaslow mirror symmetry \cite{SYZ96}, and now is studied extensively (\eg \cite{CCL20,Che21,CJY15,CXY17,CY18,HJ20,HY19,JY17,Pin19,SS19,Tak20,Tak21}).

It is also useful to consider another expression for \eqref{dHYM equation}. Let $\lambda_i$ be the eigenvalues of $\omega$ with respect to $\chi$. Then by the computation in \cite{JY17}, the equation \eqref{dHYM equation} is equivalent to
\begin{equation} \label{dHYM equation b}
\sum_{i=1}^n \arccot(\lambda_i)=\theta_0,
\end{equation}
where $\theta_0$ and $\vartheta_0$ are related as $\theta_0:=n \hpi-\vartheta_0$.

One of the main topics in the study of special Lagrangians is Thomas--Yau conjecture \cite{TY02}, which asserts that a given Lagrangian embedded in a Calabi--Yau manifold could be deformed to a special one by Hamiltonian deformations if and only if it is ``stable''. Thereafter Joyce \cite{Joy15} proposed a broad update to the Thomas-Yau conjecture in terms of Bridgeland stability and the Lagrangian mean curvature flow, and also there have been some significant progress on the symplectic side (\eg \cite{IJS16,Nev13}). However, this conjecture is still widely open. So it is natural to consider the mirror version of Thomas--Yau conjecture in the complex geometric side. Collins--Jacob--Yau \cite[Conjecture 1.4]{CJY15} predicted that the existence of the solution to \eqref{dHYM equation} is equivalent to a certain stability condition in terms of holomorphic intersection numbers for any analytic subvarieties $Y \subset X$, modeled on the Nakai--Moishezon criterion, and confirmed the conjecture for complex surfaces. Recently, Chen \cite{Che21} proved a Nakai--Moishezon type criterion for the supercritical dHYM equation (and for the $J$-equation as well) under a slightly stronger condition that these intersection numbers have a uniform positive lower bound independent of $Y$. In a recent work of Datar-Pingali \cite{DP20}, the authors extended the technique in \cite{Che21} to the case of generalized Monge-Amp\`ere and as a result, the Collins--Jacob--Yau's conjecture is confirmed when $X$ is a three dimensional projective manifold and $\tan(\vartheta_{0})\geq0$ (\ie $\cot(\theta_{0})\leq 0$) (cf. \cite[Remark 1.6]{DP20}). Jacob-Sheu \cite{JS20} showed that the conjecture holds on the blowup of complex projective space.

Meanwhile, when $\theta_0$ is very close to zero, there are similarities between the theory of the dHYM equation and the $J$-equation (\cf \cite{Che00,Don99}) for a K\"ahler form $\omega \in \a$ defined by
\[
\sum_{i=1}^n \frac{1}{\lambda_i}=\frac{n \omega^{n-1} \wedge \chi}{\omega^n}.
\]
The both are complex Hessian equations on K\"ahler manifolds and related to each other via the ``small radius limit'' as observed in \cite{CXY17,CY18}, \ie we replace $\omega$ with $s \omega$ ($s \in (0,\infty)$) and consider the limit $s \to \infty$. Then we have
\[
\lim_{s \to \infty} s \sum_{i=1}^n\arccot (s \lambda_i)=\sum_{i=1}^n \frac{1}{\lambda_i}, \quad \theta_0 \to 0.
\]
Very recently, a remarkable progress for the study of the $J$-equation was made by Song \cite{Son20}. He further extended the method of \cite{Che21}, and showed a Nakai--Moishezon type criterion for the $J$-equation without assuming a uniform lower bound for the intersection numbers, which confirms Lejmi--Sz\'ekelyhidi's conjecture \cite{LS15}.

From now on, we focus on the supercritical dHYM equation \ie $\theta_0 \in (0,\pi)$. In this case, Collins--Jacob--Yau's conjecture \cite[Conjecture 1.4]{CJY15} can be stated as follows. Actually, the direction (2) $\Rightarrow$ (1) in Conjecture \ref{CJY conjecture} only holds in the supercritical case (see \cite[Remark 1.10]{Che21}).
\begin{conj}[Conjecture 1.4 of \cite{CJY15}]\label{CJY conjecture}
Let $X$ be an $n$-dimensional compact complex manifold, $\a$ a real $(1,1)$-cohomology class, $\b$ a K\"ahler class on $X$ and $\theta_0 \in (0,\pi)$ a constant. Assume that
\[
\int_{X} \big( \Re(\a+\i \b)^n-\cot(\theta_0) \Im(\a+\i \b)^n \big) = 0.
\]
Then the following two conditions are equivalent:
\begin{enumerate}
\item for any $m$-dimensional proper analytic subvariety $Y$ of $X$ ($m=1,\ldots,n-1$), we have
\[
\int_{Y} \big( \Re(\a+\i \b)^m-\cot(\theta_0) \Im(\a+\i \b)^m \big) > 0.
\]
\item for any (or some) K\"ahler form $\chi \in \b$, there exists a unique solution to the deformed Hermitian--Yang--Mills equation \eqref{dHYM equation b}.
\end{enumerate}
\end{conj}

In this paper, we strengthen Chen's result \cite{Che21} by using the same method as \cite{Son20}. For the reader's convenience, let us recall some terminologies and notations first. Following \cite{Che21}, we say that a smooth family $\omega_{t,0}$ is a {\it test family} (emanating from a real $(1,1)$-form $\omega_0 \in \a$) if the following conditions hold:
\begin{enumerate}\setlength{\itemsep}{1mm}
\renewcommand{\theenumi}{\Alph{enumi}}
\item $\omega_{0,0}=\omega_0 \in \a$.
\item For any $t_1<t_2$, we have $\omega_{t_1,0}<\omega_{t_2,0}$.
\item There exists a number $T \geq 0$ such that $\omega_{t,0}-\cot \big( \frac{\theta_0}{n} \big) \chi>0$ holds for all $t \in [T,\infty)$.
\end{enumerate}
For instance, we have a test family by setting $\omega_{t.0}:=\omega_0+t \chi$ for any real $(1,1)$-form $\omega_0 \in \a$. For any constant $\theta_0 \in (0,\pi)$, we say that the triple $(X,\a,\b)$ is {\it stable} along a test family $\{\omega_{t,0}\}$ if for any $m$-dimensional analytic subvariety\footnote{In this paper, the terminology ``$m$-dimensional (sub)variety'' means that it is reduced, but does not have to be pure $m$-dimensional or irreducible.} $Y$ of $X$ ($m=1,\ldots,n$), we have
\begin{equation}\label{stability condition}
\begin{split}
F_{\theta_0}^{\Stab}(Y,\{\omega_{t,0}\},t)
:= {} &(\Re (\omega_{t,0}+\i \chi)^m-\cot(\theta_0) \Im (\omega_{t,0}+\i \chi)^m)\cdot Y \\
:= {} &\int_Y (\Re (\omega_{t,0}+\i \chi)^m-\cot(\theta_0) \Im (\omega_{t,0}+\i \chi)^m) \geq 0 \\
\end{split}
\end{equation}
for all $t \in [0,\infty)$, and the strict inequality holds if $m<n$ for all $t \in [0,\infty)$. If there exists $\epsilon>0$ such that for any $m$-dimensional analytic subvariety $Y$ of $X$ ($m=1,\ldots,n$), we have
\[
F_{\theta_0}^{\Stab}(Y,\{\omega_{t,0}\},t) \geq (n-m)\epsilon\int_{Y}\chi^{m}
\]
for all $t \in [0,\infty)$, then we say that the triple $(X,\a,\b)$ is {\it uniform stable} along a test family $\{\omega_{t,0}\}$.

In \cite{Che21}, Chen proved that the existence of the solution to the dHYM equation \eqref{dHYM equation b} is equivalent to the uniform stability condition.
\begin{thm} \label{Chen theorem}[Theorem 1.7 of \cite{Che21}]
Let $X$ be an $n$-dimensional compact complex manifold, $\a$ a real $(1,1)$-cohomology class, $\b$ a K\"ahler class on $X$ and $\theta_0 \in (0,\pi)$ a constant. Assume that
\begin{equation}\label{intergral condition on X}
\int_{X} \big(\Re(\a+\i \b)^n-\cot(\theta_0) \Im(\a+\i \b)^n \big) = 0.
\end{equation}
Then the following three conditions are equivalent:
\begin{enumerate}\setlength{\itemsep}{1mm}
\item the triple $(X,\a,\b)$ is uniform stable along any test family $\omega_{t,0} \in \a_t$.
\item the triple $(X,\a,\b)$ is uniform stable along some test family $\omega_{t,0} \in \a_t$.
\item for any (or some) K\"ahler form $\chi \in \b$, there exists a unique solution to the deformed Hermitian--Yang--Mills equation \eqref{dHYM equation b}.
\end{enumerate}
\end{thm}

Motivated by \cite{Son20}, it is very natural to ask if the uniform stability condition in Theorem \ref{Chen theorem} can be replaced by the stability condition alone. We give an affirmative answer to this question. The following theorem is the main theorem.
\begin{thm} \label{NM criterion A}
Let $X$ be an $n$-dimensional compact complex manifold, $\a$ a real $(1,1)$-cohomology class, $\b$ a K\"ahler class on $X$ and $\theta_0 \in (0,\pi)$ a constant. Assume that
\begin{equation}\label{intergral condition on X}
\int_{X} \big(\Re(\a+\i \b)^n-\cot(\theta_0) \Im(\a+\i \b)^n \big) = 0.
\end{equation}
Then the following three conditions are equivalent:
\begin{enumerate}\setlength{\itemsep}{1mm}
\item the triple $(X,\a,\b)$ is stable along any test family $\omega_{t,0} \in \a_t$.
\item the triple $(X,\a,\b)$ is stable along some test family $\omega_{t,0} \in \a_t$.
\item for any (or some) K\"ahler form $\chi \in \b$, there exists a unique solution to the deformed Hermitian--Yang--Mills equation \eqref{dHYM equation b}.
\end{enumerate}
\end{thm}

Theorem \ref{NM criterion A} is not exactly same as Conjecture \ref{CJY conjecture} in that it involves conditions for test families. However, we can derive the following Corollary \ref{NM criterion C} from Theorem \ref{NM criterion A}, which does not involve the assumptions of test families.

\begin{cor}\label{NM criterion C}
Let $X$ be an $n$-dimensional compact complex manifold, $\a$ a real $(1,1)$-cohomology class, $\b$ a K\"ahler class on $X$ and $\theta_0 \in (0,\pi)$ a constant. Assume that
\[
\int_{X} \big(\Re(\a+\i \b)^n-\cot(\theta_0) \Im(\a+\i \b)^n \big) = 0.
\]
Then the following two conditions are equivalent:
\begin{enumerate}
\item there exists a K\"ahler class $\gamma$ on $X$ such that for any $1\leq k\leq n$, we have
\[
\int_{X}\big(\Re(\a+\i \b)^k-\cot(\theta_0) \Im(\a+\i \b)^k\big)\wedge\gamma^{n-k} \geq 0,
\]
and for any proper $m$-dimensional analytic subvariety $Y$ of $X$ ($m=1,\ldots,n-1$) and $1\leq k\leq m$, we have
\[
\int_{Y}\big(\Re(\a+\i \b)^k-\cot(\theta_0) \Im(\a+\i \b)^k\big)\wedge\gamma^{m-k} > 0.
\]
\item for any (or some) K\"ahler form $\chi \in \b$, there exists a unique solution to the deformed Hermitian--Yang--Mills equation \eqref{dHYM equation b}.
\end{enumerate}
\end{cor}

Let us discuss the relationship between Conjecture \ref{CJY conjecture} and Corollary \ref{NM criterion C}. Roughly speaking, Conjecture \ref{CJY conjecture} predicts a kind of numerical criterion for the ``positivity" of $(m,m)$-class
\[
\Re(\a+\i \b)^m-\cot(\theta_0) \Im(\a+\i \b)^m
\]
for all $1\leq m\leq n$. This is very similar to the numerical criterion of the K\"ahler class proved by Demailly-P\u{a}un \cite{DP04}. Indeed, the assumptions of Corollary \ref{NM criterion C} are motivated by \cite[Theorem 4.2]{DP04}. In this sense, Corollary \ref{NM criterion C} can be regarded as a counterpart of \cite[Theorem 4.2]{DP04} in the supercritical dHYM setting, and also a weaker version of Conjecture \ref{CJY conjecture}. In particular, when $X$ is projective, Corollary \ref{NM criterion C} implies Conjecture \ref{CJY conjecture}.
\begin{cor}\label{NM criterion D}
If $X$ is projective, then Conjecture \ref{CJY conjecture} is true.
\end{cor}

We should mention that there are some significant differences from the arguments in \cite{Son20}. First, we consider the stability condition not only for a single class $\a$ but also for a family of cohomology classes $\a_t$ as in \cite{Che21}. Indeed, we need the notion of test families to choose the correct brunch cut of the $\arccot$ function. Also, since there is no canonical choice of test families, we check that every test family has a ``nice'' lift on the resolution of singularities of $Y$ (\cf Lemma \ref{stability on resolutions} and Lemma \ref{solvability at one}). Second, in our induction argument, we have to assume that the ambient space $X$ is always smooth, whereas \cite[Theorem 1.1]{Son20} can deal with the case when $X$ is possibly singular and embedded in a K\"ahler manifold. The reason is that we take different kinds of approaches according to the case $m<n$ and $m=n$. The former case corresponds to Theorem \ref{NM criterion B}. A key observation for Theorem \ref{NM criterion B} is that we can always take a uniformly positive twisting function in the twisted dHYM equation \eqref{twisted dHYM b} on the product space (\cf Lemma \ref{solvability of the twisted dHYM}), and this property holds only for the case $m<n$ (\cf Remark \ref{twisting function in the top dimensional case}). On the other hand, we can treat the later case essentially by the same argument as \cite[Section 5]{Che21} where the smoothness of $X$ is used (see Section \ref{completion of the proof} for more details).

The organization of the paper is as follows. First of all, we should mention that our arguments are mostly based on the seminal works \cite{Che21,Son20}. So we recommend readers to refer the both papers appropriately. Specifically, Section \ref{The twisted dHYM equation on the resolution}, \ref{gluing construction} in our paper corresponds to \cite[Section 4, 6]{Son20}, and Section \ref{the twisted dHYM equation on the product space}, \ref{the mass concentration and local smoothing} corresponds to \cite[Section 5]{Son20} respectively. Our Section \ref{The twisted dHYM equation on the resolution}, \ref{the twisted dHYM equation on the product space}, \ref{the mass concentration and local smoothing}, \ref{gluing construction}, \ref{completion of the proof} are also based on \cite[Section 5]{Che21}.

In Section \ref{preliminaries}, we define notations and prove basic properties for functions on the space of Hermitian matrices. Then we recall some results for subsolutions proved in \cite{Che21,CJY15}. In the end of the section, we will give the proof for the easier part (1) $\Rightarrow$ (2), (3) $\Rightarrow$ (1) of Theorem \ref{NM criterion A}. We also show how to apply Theorem \ref{NM criterion A} to prove Corollary \ref{NM criterion C}, and apply Corollary \ref{NM criterion C} to prove Corollary \ref{NM criterion D}. In the remaining part of the paper, we will consider the harder part (2) $\Rightarrow$ (3) of Theorem \ref{NM criterion A}, in particular, Theorem \ref{NM criterion B}. In Section \ref{The twisted dHYM equation on the resolution}, we consider the twisted dHYM equation on the resolution of singularities $\hat{Y}$ for any proper analytic subvariety $Y \subset X$. Then we also study the twisted dHYM equation on the product space $\cY=\hat{Y} \times \hat{Y}$ in Section \ref{the twisted dHYM equation on the product space}. In Section \ref{the mass concentration and local smoothing}, we establish some estimates for the fiber integration by using the mass concentration argument, and construct a subsolution $\Omega$ as a current. We show that on each coordinate ball of a sufficiently finer finite covering of $\hat{Y}$, the smoothing of $\Omega$ for sufficiently small scale satisfies the condition for subsolutions away from the exceptional set $E_0 \subset \hat{Y}$ on which $\chi$ degenerates. In Section \ref{gluing construction}, we glue these local subsolutions together by taking the regularized maximum. In Section \ref{completion of the proof}, we give a proof of the harder part (2) $\Rightarrow$ (3) of Theorem \ref{NM criterion A} based on the argument in \cite[Section 5]{Che21}, and accomplish the proof of Theorem \ref{NM criterion A}.
\begin{ackn}
M.-C. Lee was partially supported by NSF grant DMS-1709894 and EPSRC grant number P/T019824/1. R.~Takahashi was supported by Grant-in-Aid for Early-Career Scientists (20K14308) from JSPS.
\end{ackn}

\section{Preliminaries} \label{preliminaries}
By following \cite[Section 5]{Che21}, for any positive integers $k \leq n$, we define functions $\cP_{k,n}$, $\cQ_{k,n}$ on $\R^n$ by
\[
\cP_{k,n}(\lambda_1,\ldots,\lambda_n):=\max_{I} \sum_{i \in I} \arccot (\lambda_i),
\]
\[
\cQ_{k,n}(\lambda_1,\ldots,\lambda_n):=\max_{J}\sum_{j \in J} \arccot (\lambda_j),
\]
where $\cP_{1,n}:=0$ and the maximum is taken over all subsets $I \subset \{1,\ldots,n\}$ (resp. $J \subset \{1,\ldots,n\}$) of cardinality $k-1$ (resp. $k$).\footnote{Although we get immidiately $\cP_{k+1,n}=\cQ_{k,n}$ by the definition, we will use the separate notations to avoid confusion.}
Let $\Gamma$ be the set which consists of all $\lambda \in \R^n$ satisfying $\cQ_{n,n}(\lambda)<\pi$.
\begin{prop} \label{monotone and concave}
We have the following:
\begin{enumerate}\setlength{\itemsep}{1mm}
\item The functions $\cP_{k,n}$ and $\cQ_{k,n}$ are monotone on $\R^n$, \ie if $(\lambda_1,\ldots,\lambda_n)$, $(\lambda'_1,\ldots,\lambda'_n) \in \R^n$ satisfies $\lambda_i \leq \lambda'_i$ for all $i=1,\ldots,n$, then we have
\[
\cP_{k,n}(\lambda'_1,\ldots,\lambda'_n) \leq \cP_{k,n}(\lambda_1,\ldots,\lambda_n), \quad \cQ_{k,n}(\lambda'_1,\ldots,\lambda'_n) \leq \cQ_{k,n}(\lambda_1,\ldots,\lambda_n).
\]
\item If $k \leq k'$, then
\[
\cP_{k,n} \leq \cP_{k',n}, \quad \cQ_{k,n} \leq \cQ_{k',n}
\]
hold on $\R^n$.
\item $\Gamma$ is a convex subset of $\R^n$ whose boundary is a smooth hypersurface.
\item The functions $\cot(\cP_{k,n}(\cdot))$ and $\cot(\cQ_{k,n}(\cdot))$ are concave on $\Gamma$.
\end{enumerate}
\end{prop}
\begin{proof}
(1) and (2) are obvious, and (3) follows from \cite[Lemma 2.1]{Yua06}. For (4), the concavity of the function $\cQ_{n,n}$ on $\Gamma$ was proved in the course of the proof of \cite[Theorem 1.1]{Tak20}. By using this fact, one can observe that
\[
\cot(\cP_{k,n}(\lambda_1,\ldots,\lambda_n))=\min_I \cot \big( \sum_{i \in I} \arccot(\lambda_i) \big)
\]
and each $\cot \big( \sum_{i \in I} \arccot(\lambda_i) \big)$ is concave on $\Gamma$. This shows that $\cot(\cP_{k,n}(\cdot))$ is concave. The proof for $\cot(\cQ_{k,n}(\cdot))$ is similar.
\end{proof}
Let $\Herm(n)$ be the space of all $n \times n$ Hermitian matrices and $A, B \in \Herm(n)$. We regard $A, B$ as Hermitian forms on $\C^n$. Assume that $A$ is positive definite and let $\lambda_i$ be the eigenvalues of the matrix $A^{-1} B$. Then $\cP_{A,k,n}(B)$ (resp. $\cQ_{A,k,n}(B)$) is defined as $\cP_{k,n}(\lambda_1,\ldots,\lambda_n)$ (resp. $\cQ_{k,n}(\lambda_1,\ldots,\lambda_n)$). Since the eigenvalues of $A^{-1}B$ is invariant under linear transformations, we may always assume that $A=I_n$ and $B$ is diagonal. Let $\Gamma_A \subset \Herm(n)$ be the subset of $n \times n$ Hermitian matrices $B$ satisfying $\cQ_{A,n,n}(B)<\pi$. Then by Proposition \ref{monotone and concave}, we know that:
\begin{prop} \label{monotone and concave matrix}
We have the following:
\begin{enumerate}\setlength{\itemsep}{1mm}
\item The functions $\cP_{A,k,n}(\cdot)$ and $\cQ_{A,k,n}(\cdot)$ are monotone on $\Herm(n)$.
\item If $k, k' \in \N$ satisfy $k \leq k'$, then
\[
\cP_{A,k,n} \leq \cP_{A,k',n}, \quad \cQ_{A,k,n} \leq \cQ_{A,k',n}
\]
hold on $\Herm(n)$.
\item $\Gamma_A$ is a convex subset of $\Herm(n)$ whose boundary is a smooth hypersurface.
\item $\cot(\cP_{A,k,n}(\cdot))$ and $\cot(\cQ_{A,k,n}(\cdot))$ are concave on $\Gamma_A$.
\end{enumerate}
\end{prop}
\begin{proof}
First we assume $A=I_n$ and let $\lambda_1 \leq \ldots \leq \lambda_n$ (resp. $\lambda'_1 \leq \ldots \leq \lambda'_n$) be the eigenvalues of an $n \times n$ Hermitian matrix $B$ (resp. $B'$). Then it is well known that the condition $B \leq B'$ implies $\lambda_i \leq \lambda_i'$ for all $i=1,\ldots,n$. Thus the statement (1) follows from the monotonicity of the functions $\cP_{k,n}$, $\cQ_{k,n}$. (2) is clear, and (3) follows from \cite[Lemma 2.1]{Yua06}. As for (4), it is well-known (cf. \cite{And94,Ger96,Spr05}) that if the function $f(\lambda_1,\ldots,\lambda_n)$ is symmetric and concave, then the associated function $F_A(B):=f(\lambda_1,\ldots,\lambda_n)$ in terms of a symmetric function of the eigenvalues $\lambda_i$ is concave. Thus the function $\cQ_{A,n,n}$ is concave on $\Gamma_A$. In general case, the function $f(\cdot):=\cot(\cP_{k,n}(\cdot))$ (or $f(\cdot):=\cot(\cQ_{k,n}(\cdot))$) may not be differentiable at points $(\lambda_1,\ldots,\lambda_n)$ where $\lambda_i$ are not distinct. However, we can approximate $f$ with a sequence of smooth symmetric concave functions. Indeed, for any $B, B' \in \Gamma_A$, we take a small constant $\e>0$ such that $B, B' \in \Gamma_A^\e$, where the convex subset $\Gamma_A^\e \Subset \Gamma_A$ is given by $\Gamma_A^\e:=\{ B'' \in \Gamma_A|\cQ_{A,n,n}(B'')<\pi-\e\}$. Also let $\Gamma^\e:=\{ \lambda \in \Gamma^\e |\cQ_{n,n}(\lambda)<\pi-\e\}$ be the associated convex subset of $\Gamma^\e$. Since the derivative of $\cQ_{n,n}$ is uniformly bounded on $\Gamma$, we know that there exists $r_0>0$ such that for any $r \in (0,r_0)$ and $\lambda \in \Gamma^\e$, we have $B_r(\lambda) \Subset \Gamma$, where $B_r(\lambda)$ denotes a Euclidean ball of radius $r$ centered at $\lambda$. Then for any $r \in (0,r_0)$, we define the smoothing of $f$ as in Definition \ref{local smoothing} and denote it by $f^{(r)}(\lambda)$ ($\lambda \in \Gamma^\e$). We note that $f^{(r)}$ converges uniformly to $f$ as $r \to 0$ on any compact subset of $\Gamma^\e$. Also one can easily see that from the definition of the smoothing and concavity, $f^{(r)}$ is concave on $\Gamma^\e$ since $f$ is so by Proposition \ref{monotone and concave} (4). Let $\mathfrak{S}_n$ denote the symmetric group which acts linearly on $\Gamma^\e \subset \R^n$ by $\lambda:=(\lambda_1,\ldots,\lambda_n) \mapsto \lambda_\s:=(\lambda_{\s(1)},\ldots,\lambda_{\s(n)})$. Then the symmetrization $f_r(\lambda):=\frac{1}{n!} \sum_{\s \in \mathfrak{S}_n} f^{(r)}(\lambda_\s)$ is still concave and converges to $f$ on $\Gamma^\e$ as $r \to 0$ since $f^{(r)} \to f$ and $f$ is symmetric. Let $F_{A,r}$ denote the associated function with respect to $f_r$ on $\Gamma_A^\e$. Then we have $F_{A,r} \to F_A$ as $r \to 0$. We can apply the result \cite{Spr05} to each $f_r$ to know that
\[
F_{A,r}(sB+(1-s)B') \geq s F_{A,r}(B)+(1-s)F_{A,r}(B')
\]
for any $s \in [0,1]$. By taking the limit $r \to 0$, we have
\[
F_A(sB+(1-s)B') \geq s F_A(B)+(1-s)F_A(B')
\]
as desired.
\end{proof}

It is also sometimes useful to consider the variational characterization of $\cP_{A,k,n}$ and $\cQ_{A,k,n}$. Let $A=I_n$ be the identity matrix for simplicity and $B \in \Herm(n)$. For any $(k-1)$-dimensional subspace $V$ of $\C^n$, there exists $n \times (k-1)$ matrix $U \in \Mat(n,k-1;\C)$ such that $\overline{U}^T U=I_{k-1}$ and a basis of $V$ consists of the columns of $U$. We choose $U$ such that $\overline{U}^T B U=\diag(\lambda_1',\ldots,\lambda_{k-1}')$ and $S \in U(n)$ such that $S_{i \bar{j}}=U_{i \bar{j}}$ ($i=1,\ldots,n; j=1,\ldots,k-1$). Then the Schur-Horn theorem says that the diagonal $(\lambda_1',\ldots,\lambda_{k-1}', (\overline{S}^T B S)_{k \bar{k}},\ldots,(\overline{S}^T B S)_{n \bar{n}})$ of the matrix $\overline{S}^T B S$ lies in the convex full of the vectors obtained by permuting the entries of $(\lambda_1,\ldots,\lambda_n)$, where $\lambda_i$ are eigenvalues of $B$. Thus by the concavity of the function $\cot(\cP_{k,n}(\cdot))$ on $\Gamma$, we know that
\begin{equation} \label{variational characterization for P}
\cP_{I_n,k,n}(B)=\max_I \sum_{i \in I} \arccot(\lambda_i)=\max_{U \in \Mat(n,k-1;\C), \; \overline{U}^T U=I_{k-1}, \; \overline{U}^T B U=\diag(\lambda_1',\ldots,\lambda_{k-1}')} \sum_{i=1}^{k-1} \arccot(\lambda_i')
\end{equation}
for all $B \in \Gamma_{I_n}$. A similar observation shows that
\begin{equation} \label{variational characterization for Q}
\cQ_{I_n,k,n}(B)=\max_J \sum_{i \in J} \arccot(\lambda_i)=\max_{K \in \Mat(n,k;\C), \; \overline{K}^T K=I_k, \; \overline{K}^T B K=\diag(\lambda_1'',\ldots,\lambda_k'')} \sum_{i=1}^k \arccot(\lambda_i'')
\end{equation}
for all $B \in \Gamma_{I_n}$. These are simple extensions of formulas obtained in \cite[Section 5]{Che21}.

Let $X$ be an $n$-dimensional K\"ahler manifold $X$, $\a$ a real $(1,1)$-cohomology class and $\b$ a K\"ahler class on $X$. For any closed real $(1,1)$-form $\omega \in \a$ and K\"ahler form $\chi \in \b$, we can also define functions $\cP_{\chi,k,n}(\omega)$, $\cQ_{\chi,k,n}(\omega)$ in a standard manner. Let $0<\theta_0 \leq \Theta_0 \leq \pi$ be constants and $Y$ an analytic subvariety of $X$. Let $\omega_0 \in \a$ be a fixed real $(1,1)$-form on $X$. We define the set $\Gamma_{\chi, \a, \theta_0, \Theta_0}(Y,X)$ which consists of germs of real closed $(1,1)$-forms $\omega \in \a|_{Y,X}$ satisfying $\cP_{\chi,n,n}(\omega)<\theta_0$ and $\cQ_{\chi,n,n}(\omega)<\Theta_0$ on $Y \subset X$, where the class $\a|_{Y,X}$ consists of all germs of real closed $(1,1)$-forms $\omega$ at $Y \subset X$ such that $\omega=\omega_0+\dd \phi_\omega$ for some smooth function $\phi_\omega$ defined on a neighborhood of $Y$ in $X$. Then one can prove the following lemma exactly in the same way as \cite[Lemma 8.2]{CJY15}:
\begin{lem} \label{positivity of forms}
For any $\omega \in \Gamma_{\chi, \a, \theta_0, \pi}(Y,X)$, on a neighborhood of $Y$ in $X$, we have
\[
\Im (e^{-\i \theta_0} (\omega+\i \chi)^p)<0
\]
for all $p=1,\ldots,n-1$, where negativity is to be understood in the sense of $(p,p)$-forms (\cf \cite[Chapter III, Section 1.A]{Dem12}).
\end{lem}

\begin{proof}
Actually, this lemma is equivalent to \cite[Lemma 8.2]{CJY15}. Since we use different notations, for the reader's convenience, we include a brief proof here. It suffices to show that for any simple positive $(n-k,n-k)$-form $\Omega$,
\[
\Im\left(e^{-\sqrt{-1}\theta_0}(\omega+\sqrt{-1}\chi)^{k}\wedge\Omega\right) < 0.
\]
We choose a coordinates such that $\chi_{i\bar{j}} = \delta_{ij}$ and $\omega_{i\bar{j}} = \lambda_{i}\delta_{ij}$. Then
\[
\begin{split}
& (\omega+\sqrt{-1}\chi)^{k} = (\sqrt{-1})^{k}k!\sum_{J}\prod_{j\in J}(\lambda_{i}+\sqrt{-1})dz^{J}\wedge d\bar{z}^{J} \\
= {} & (\sqrt{-1})^{k}k!\sum_{J}e^{\sqrt{-1}\sum_{j\in J}\arccot(\lambda_{i})}\prod_{j\in J}\left(\sqrt{\lambda_{i}^{2}+1}\right)dz^{J}\wedge d\bar{z}^{J}.
\end{split}
\]
where the sum is over all subsets $J\subset\{1,\ldots,n\}$ such that $|J|=k$. The form $\Omega$ can be written as
\[
\Omega = (\sqrt{-1})^{n-k}\sum_{J}c_{J}dz^{J^{c}}\wedge d\bar{z}^{J^{c}}+\Omega',
\]
where $J^{c}$ denotes the complement of $J$ and $\Omega'$ is smooth form such that $\Omega'\wedge dz^{J}\wedge d\bar{z}^{J}=0$ for all $J$. Since $\Omega$ is a simple positive form, then $c_{J}\geq0$ for all $J$ and at least one $c_{J}$ is positive. Since $P_{\chi}(\omega)<\theta$ and $|J|=k\leq n-1$, then $\sum_{j\in J}\arccot(\lambda_{i})-\theta<0$ for all $J$. Therefore,
\[
\begin{split}
& \frac{\Im\left(e^{-\sqrt{-1}\theta}(\alpha+\sqrt{-1}\chi)^{k}\wedge\Omega\right)}{\chi^{n}} \\
= {} & \frac{k!}{n!}\sum_{J}c_{J}\Im\left(e^{\sqrt{-1}\left(\sum_{j\in J}\arccot(\lambda_{i})-\theta\right)}\right)\prod_{j\in J}\left(\sqrt{\lambda_{i}^{2}+1}\right)
< 0.
\end{split}
\]
\end{proof}

If $Y$ is smooth and $Y=X$, we often write $\cP_{\chi,n,n}(\omega)$, $\cQ_{\chi,n,n}(\omega)$, $\Gamma_{\chi,\a,\theta_0, \Theta_0}(Y,Y)$ as $P_\chi(\omega)$, $Q_\chi(\omega)$, $\Gamma_{\chi,\a,\theta_0, \Theta_0}(Y)$ for simplicity. We remark that if $\chi\leq \tilde\chi $ and $\omega$ is K\"ahler, then we have
\begin{equation} \label{positive definite monotonicity}
\cP_{\chi,k,n}(\omega) \leq \cP_{\tilde \chi,k,n}(\omega), \quad \cQ_{\chi,k,n}(\omega) \leq \cQ_{\tilde \chi,k,n}(\omega).
\end{equation}
This property does not hold for general $(1,1)$-form $\omega$, especially when $\omega$ is not K\"ahler. However, we can show the following uniform semi-continuity when $k=n$, which is similar to \cite[Definition 5.10]{Che21}.

\begin{prop}\label{semi-continuity}
For $0<\theta<\pi$, there is $c_0(n,\theta)>0$ such that the following holds. Suppose $\omega,\chi$ are closed real $(1,1)$-forms where $\chi$ is K\"ahler, then for all $\varepsilon<\theta$,
\begin{enumerate}\setlength{\itemsep}{1mm}
\item[(a)] $Q_{\chi}(\omega+\varepsilon \chi)< \theta-c_0\varepsilon$ if $Q_{\chi}(\omega)<\theta$;
\item[(b)] $P_{\chi}(\omega+\varepsilon \chi)< \theta-c_0\varepsilon$ if $P_{\chi}(\omega)<\theta$.
\end{enumerate}
\end{prop}
\begin{proof}
We prove $(a)$ first. We let $\lambda_1\leq \cdots\leq \lambda_n$ be the eigenvalues of $\omega$ with respect to $\chi$. If $Q_{\chi}(\omega)\leq \frac12 \theta$, then the conclusion holds directly. So it suffices to consider the case $\frac12 \theta< Q_{\chi}(\omega)<\theta$. In this case, we have $|\lambda_1|\leq C(n,\theta)$ and hence,
\begin{equation*}
\begin{split}
Q_{\chi}(\omega+\varepsilon \chi)&= \sum_{i=1}^n \arccot(\lambda_i+\varepsilon)\\
&\leq  \sum_{i=2}^n \arccot(\lambda_i)+\arccot(\lambda_1+\varepsilon) \\[3mm]
&\leq Q_{\chi}(\omega)-c_0(n,\theta) \varepsilon.
\end{split}
\end{equation*}
Here we have applied mean value theorem on the function $\arccot(x)$ and used the fact that $\lambda_1$ is uniformly bounded.  The conclusion (b) can be proved analogously.
\end{proof}

\begin{prop}\label{uniform continuity}
For $\theta<\pi$, there is $\sigma_0(n,\theta)>0$ such that the following holds: let $\chi_1,\chi_2$ and $\chi_3$ are K\"ahler and $\omega$ be a closed real $(1,1)$ form which satisfies $\chi_1,\chi_2\leq 4\chi_3$ and $|\chi_1-\chi_2|\leq \sigma^5 \chi_3$ for some $\sigma<\sigma_0$. Then
\begin{enumerate}\setlength{\itemsep}{1mm}
\item[(a)] $Q_{\chi_1}(\omega+\sigma \chi_3)<\theta$ if $Q_{\chi_2}(\omega)<\theta$;
\item[(b)] $P_{\chi_1}(\omega+\sigma \chi_3)<\theta$ if $P_{\chi_2}(\omega)<\theta$.
\end{enumerate}
\end{prop}
\begin{proof}
We will focus on proving (a) while (b) can be proved using similar argument. It will be sufficient to prove $Q_{\chi_1}(\omega+\sigma \chi_3)\leq \theta+\sigma$ under $Q_{\chi_2}(\omega)<\theta$ and $|\chi_1-\chi_2|\leq \sigma^4 \chi_3$ since we can use the fact that $\chi_1\leq 4\chi_3$ and apply Proposition~\ref{semi-continuity} to show that for $\sigma_0$ sufficiently small,
\[
Q_{\chi_1}(\omega+5\sigma \chi_3+5c_0^{-1}\sigma \chi_3)
< Q_{\chi_1}(\omega+\sigma \chi_3+c_0^{-1}\sigma \chi_1)
\leq (\theta+\sigma) -c_0 \left( c_0^{-1}\sigma\right)=\theta.
\]
The result will follow by shrinking $\sigma_0$ further.

To show $Q_{\chi_1}(\omega+\sigma \chi_3)\leq \theta+\sigma$, it is equivalent to show that
\[
\Im\left( e^{-\sqrt{-1}(\theta+\sigma)}(\omega+\sigma\chi_3+\sqrt{-1}\chi_1)^n \right) \leq 0.
\]
We will assume $\sigma_0<1$. For notational convenience, we let $\theta'=\theta+\sigma$ and $\chi_1=\chi_2+\chi_d$. Then
\[
\begin{split}
&\quad \Im\left( e^{-\sqrt{-1}(\theta+\sigma)}(\omega+\sigma\chi_3+\sqrt{-1}\chi_1)^n   \right)\\[1mm]
&= \Im\left( e^{-\sqrt{-1}\theta'}\left[(\omega+\sqrt{-1}\chi_2) +(\sigma \chi_3+\sqrt{-1} \chi_d)\right]^n   \right)\\[1mm]
&= \Im\left( e^{-\sqrt{-1}\theta'} (\omega+\sqrt{-1}\chi_2)^{n} \right)\\
&+\sum_{k=1}^n \binom{n}{k}\Im\left( e^{-\sqrt{-1}\theta'} (\omega+\sqrt{-1}\chi_2)^{n-k} \wedge (\sigma \chi_3+\sqrt{-1} \chi_d)^k  \right).
\end{split}
\]
By using $Q_{\chi_2}(\omega)<\theta<\theta'$, the first term is non-positive. We will show that each term in the partial sum is non-positive. For $1\leq k\leq n$,
\begin{equation*}
\begin{split}
&\quad \Im\left( e^{-\sqrt{-1}\theta'} (\omega+\sqrt{-1}\chi_2)^{n-k} \wedge (\sigma \chi_3+\sqrt{-1} \chi_d)^k   \right)\\
&=\sum_{l=0}^{k-1}\binom{k}{l}\Im\left( e^{-\sqrt{-1}\theta'} (\omega+\sqrt{-1}\chi_2)^{n-k} \wedge \left[(\sigma \chi_3)^l \wedge (\sqrt{-1}\chi_d)^{k-l} \right]   \right)\\
&\quad + \Im \left( e^{-\sqrt{-1}\theta'} (\omega+\sqrt{-1}\chi_2)^{n-k} \right)\wedge (\sigma\chi_3)^k \\[1mm]
&=\mathbf{S}+\mathbf{G}.
\end{split}
\end{equation*}
We now show that if $\sigma_0$ is sufficiently small, then $\mathbf{S}$ will be dominated by the negative term $\mathbf{G}$. We first show that  $e^{-\sqrt{-1}\theta'} (\omega+\sqrt{-1}\chi_2)^{k}$ is controlled by its imaginary part due to the squeezed angle.
\begin{cla}\label{im-control-real}
We have
\begin{equation*}
\Re\left(e^{-\sqrt{-1}\theta'} (\omega+\sqrt{-1}\chi_2)^{k}  \right) \leq -\frac2{\sigma} \Im\left(e^{-\sqrt{-1}\theta'} (\omega+\sqrt{-1}\chi_2)^{k}  \right).
\end{equation*}
\end{cla}
\begin{proof}[Proof of Claim \ref{im-control-real}]
To show this, we choose a coordinate such that $(\chi_2)_{i\bar j}=\delta_{ij}$ and $\omega_{i\bar j}=\lambda_i \delta_{ij}$. Then
\begin{equation*}
\begin{split}
&\quad \Re\left(e^{-\sqrt{-1}\theta'} (\omega+\sqrt{-1}\chi_2)^{k}  \right) \\
&=k! \sum_J \left[ \cos \left(\theta'-\sum_{j\in J} \arccot(\lambda_i) \right)\prod_{j\in J}\sqrt{\lambda_i^2+1}\right] (\sqrt{-1})^k dz^J \wedge d\bar z^J
\end{split}
\end{equation*}
where the sum is taken over all subset $J\subset \{1,...,n\}$ so that $|J|=k$. Similarly,
\begin{equation*}
\begin{split}
&\quad -\Im\left(e^{-\sqrt{-1}\theta'} (\omega+\sqrt{-1}\chi_2)^{k}  \right) \\
&=k! \sum_J \left[ \sin \left(\theta'-\sum_{j\in J} \arccot(\lambda_i) \right)\prod_{j\in J}\sqrt{\lambda_i^2+1}\right] (\sqrt{-1})^k dz^J \wedge d\bar z^J.
\end{split}
\end{equation*}
Recalling that $Q_{\chi_2}(\omega)<\theta$ and $\theta'=\theta+\sigma$, we have
$$\theta'-\sum_{j\in J} \arccot(\lambda_i)=\theta-\sum_{j\in J} \arccot(\lambda_i)+\sigma\geq \sigma.$$
Hence, the elementary inequality $\sin x\geq \frac12 x$ for $x<\sigma$ implies
\begin{equation*}
\begin{split}
-\Im\left(e^{-\sqrt{-1}\theta'} (\omega+\sqrt{-1}\chi_2)^{k}  \right) &\geq \frac12 \sigma \cdot k! \sum_J \left[ \prod_{j\in J}\sqrt{\lambda_i^2+1}\right] (\sqrt{-1})^k dz^J \wedge d\bar z^J\\
&\geq \frac12\sigma \Re\left(e^{-\sqrt{-1}\theta'} (\omega+\sqrt{-1}\chi_2)^{k}  \right).
\end{split}
\end{equation*}
This proves Claim \ref{im-control-real}.
\end{proof}

Now we compare $\mathbf{S}$ and $\mathbf{G}$.
\begin{cla}\label{terms S and G}
For all $0\leq l\leq k-1$, we have
\[\begin{split}
&\quad \Im\left( e^{-\sqrt{-1}\theta'} (\omega+\sqrt{-1}\chi_2)^{n-k} \wedge \left[(\sigma \chi_3)^l \wedge (\sqrt{-1}\chi_d)^{k-l} \right]   \right) \\
&\leq -2\sigma^{k+2}\Im\left(e^{-\sqrt{-1}\theta'}(\omega+\sqrt{-1}\chi_2)^{n-k} \right) \wedge \chi_3^k.
\end{split}\]
\end{cla}
\begin{proof}[Proof of Claim \ref{terms S and G}]
We first consider the case when $k-l$ is even. In this case, $\left(\sqrt{-1} \chi_d \right)^{k-l}$ is real and hence,
\[
\begin{split}
&\quad \Im\left( e^{-\sqrt{-1}\theta'} (\omega+\sqrt{-1}\chi_2)^{n-k} \wedge \left[(\sigma \chi_3)^l \wedge (\sqrt{-1}\chi_d)^{k-l} \right]   \right) \\
&= \pm
\Im\left( e^{-\sqrt{-1}\theta'} (\omega+\sqrt{-1}\chi_2)^{n-k}  \right) \wedge (\sigma \chi_3)^l \wedge \chi_d ^{k-l}\\
&\leq - \sigma^{4k-3l}\cdot
\Im\left( e^{-\sqrt{-1}\theta'} (\omega+\sqrt{-1}\chi_2)^{n-k}  \right) \wedge \chi_3^k\\
&\leq - \sigma^{k+3}\cdot
\Im\left( e^{-\sqrt{-1}\theta'} (\omega+\sqrt{-1}\chi_2)^{n-k}  \right) \wedge \chi_3^k,
\end{split}
\]
where we have used $-\sigma^4\chi_3\leq \chi_d\leq \sigma^4\chi_3$ and $\Im\left( e^{-\sqrt{-1}\theta'} (\omega+\sqrt{-1}\chi_2)^{n-k} \right)<0$ (cf. Lemma \ref{positivity of forms}).

\bigskip

When $k-l$ is odd, we have $(\sqrt{-1}\chi_d)^{k-l}$ is imaginary and hence, by Claim~\ref{im-control-real},
\[
\begin{split}
&\quad \Im\left( e^{-\sqrt{-1}\theta'} (\omega+\sqrt{-1}\chi_2)^{n-k} \wedge \left[(\sigma \chi_3)^l \wedge (\sqrt{-1}\chi_d)^{k-l} \right]   \right) \\
&= \pm
\Re\left( e^{-\sqrt{-1}\theta'} (\omega+\sqrt{-1}\chi_2)^{n-k}  \right) \wedge (\sigma \chi_3)^l \wedge \chi_d ^{k-l}\\
&\leq -2\sigma^{3(k-l)+k-1}\cdot\Im\left( e^{-\sqrt{-1}\theta'} (\omega+\sqrt{-1}\chi_2)^{n-k}  \right) \wedge \chi_3^k\\
&\leq  -2\sigma^{2+k}\cdot\Im\left( e^{-\sqrt{-1}\theta'} (\omega+\sqrt{-1}\chi_2)^{n-k}  \right) \wedge \chi_3^k,
\end{split}
\]
where we have used $-\sigma^4\chi_3\leq \chi_d\leq \sigma^4\chi_3$ and $\Im\left( e^{-\sqrt{-1}\theta'} (\omega+\sqrt{-1}\chi_2)^{n-k} \right)<0$ (cf. Lemma \ref{positivity of forms}) again.
\end{proof}

By using Claim \ref{terms S and G}, we arrive at the following:
\[
\begin{split}
\mathbf{G}+\mathbf{S}&\leq \sigma^k\left(1-C_n\sigma^{2} \right)\cdot \Im\left( e^{-\sqrt{-1}\theta'} (\omega+\sqrt{-1}\chi_2)^{n-k}  \right) \wedge \chi_3^k.
\end{split}
\]
Therefore, if $\sigma_0$ is small enough, we have $\mathbf{G}+\mathbf{S}<0$ for all $1\leq k\leq n$. This completes the proof.
\end{proof}

Following \cite{CJY15}, we would also like to call the element of $\Gamma_{\chi,\a,\theta_0, \Theta_0}(X)$ as subsolutions. Indeed, it was proved in \cite{CJY15} that the existence of subsolutions leads to the existence of genuine solutions to the dHYM equation \eqref{dHYM equation b}. We will use the following $f$-twisted version obtained by Chen \cite[Proposition 5.5]{Che21}:
\begin{prop} \label{continuity method}
Let $X$ be an $n$-dimensional compact complex manifold, $\chi$ a K\"ahler form, $\a$ a real $(1,1)$-cohomology class and $0<\theta_0<\Theta_0<\pi$ constants. Then there exists a constant $c>0$ depending only on $n,\theta_0, \Theta_0$ such that the following statement holds. Assume the following:
\begin{enumerate}
\item When $n=1,2,3$, $f \geq 0$ is a constant satisfying
\[
\int_X f \chi^n=\int_X (\Re(\omega_0+\i \chi)^n-\cot(\theta_0) \Im (\omega_0+\i \chi)^n) \geq 0.
\]
\item When $n \geq 4$, $f>-c$ is a smooth function satisfying
\[
\int_X f \chi^n=\int_X (\Re(\omega_0+\i \chi)^n-\cot(\theta_0) \Im (\omega_0+\i \chi)^n) \geq 0.
\]
\item $\Gamma_{\chi, \a, \theta_0, \Theta_0}(X) \neq \emptyset$.
\end{enumerate}
Then there exists a solution $\omega$ to the twisted dHYM equation
\begin{equation} \label{f twisted dHYM equation}
\Re(\omega+\i \chi)^n-\cot(\theta_0) \Im (\omega+\i \chi)^n-f \chi^n=0,
\end{equation}
where $\omega \in \Gamma_{\chi, \a, \theta_0, \Theta_0}(X)$.
\end{prop}
\begin{rk}
In the above proposition, a smooth function $f$ is allowed to be slightly negative, which is crucial in the argument of \cite{Che21}, but is not used in the proof of Theorem \ref{NM criterion B}. Also it is important to note that we cannot take $\theta_0$ and $\Theta_0$ as the same constant. Indeed, if $f$ is negative at some point $z \in X$ and there is a solution $\omega \in \Gamma_{\chi, \a, \theta_0, \theta_0}(X)$, then the equation \eqref{f twisted dHYM equation} automatically yields that $Q_\chi(\omega(z))>\theta_0$.
\end{rk}
Now we are ready to prove the easy part of Theorem \ref{NM criterion A}.
\begin{proof}[Proof of (1) $\Rightarrow$ (2), (3) $\Rightarrow$ (1) in Theorem \ref{NM criterion A}]
(1) $\Rightarrow$ (2) is trivial. So we will show (3) $\Rightarrow$ (1). Assume that there exists the solution $\omega_0$ to \eqref{dHYM equation b}. Then we have $\omega_0 \in \Gamma_{\chi, \a, \theta_0,\pi}(X)$. By Lemma \ref{positivity of forms} and the assumption of Theorem \ref{NM criterion A}, for any test family $\omega_{t,0} \in \a_t$ ($t \in [0,\infty)$) emanating from $\omega_0$ and $m$-dimensional analytic subvariety $Y$ of $X$, we have
\[
F_{\theta_0}^{\Stab}(Y,\{\omega_{t,0}\},0) \geq 0,
\]
and the strict inequality holds if $m<n$. Since $\omega_{t,0} \geq \omega_0$ for all $t \in [0,\infty)$, the monotonicity of $P_\chi$, $Q_\chi$ implies that $\omega_{t,0} \in \Gamma_{\chi, \a_t, \theta_0,\pi}(X)$ for all $t \in [0,\infty)$. Thus by using Lemma \ref{positivity of forms} again, we observe that
\[
\ddt F_{\theta_0}^{\Stab}(Y,\{\omega_{t,0}\},t)=m\int_Y \ddt \omega_{t,0} \wedge \big( \Re(\omega_{t,0}+\i \chi)^{m-1}-\cot(\theta_0) \Im (\omega_{t,0}+\i \chi)^{m-1} \big) \geq 0
\]
for all $t \in [0,\infty)$. So the triple $(X,\a,\b)$ is stable along any test family $\omega_{t,0}$ emanating from $\omega_0$. Finally, we remark that for any $\omega_0' \in \a$ and test family $\omega_{t,0}' \in \a_t$ ($t \in [0,\infty)$) emanating from $\omega_0'$, we can always obtain a test family $\omega_{t,0}$ ($t \in [0,\infty)$) emanating from $\omega_0$ by setting
\[
\omega_{t,0}:=\omega_0-\omega_0'+\omega_{t,0}' \in \a_t,
\]
and we have
\[
F_{\theta_0}^{\Stab}(Y,\{\omega_{t,0}'\},t)=F_{\theta_0}^{\Stab}(Y,\{\omega_{t,0}\},t)
\]
for all $t \in [0,\infty)$ by the cohomological invariance. This completes the proof.
\end{proof}

Now we are in a position to prove Corollary \ref{NM criterion C} by Theorem \ref{NM criterion A}.
\begin{proof}[Proof of Corollary \ref{NM criterion C}]
(2) $\Rightarrow$ (1) follows from Lemma \ref{positivity of forms}. Next we will apply Theorem \ref{NM criterion A} to show (1) $\Rightarrow$ (2). Let $\sigma\in\gamma$ be a K\"ahler form on $X$. It is clear that
\[
\omega_{t,0} := \omega_{0}+t\sigma, \quad t \in [0,\infty)
\]
is a test family emanating from $\omega_{0}$. By Theorem \ref{NM criterion A}, it suffices to show that the triple $(X,\alpha,\beta)$ is stable along this test family. For any $m$-dimensional analytic subvariety $Y$ of $X$ ($m=1,\ldots,n$) and $t \in [0,\infty)$, we compute
\begin{eqnarray*}
&& \int_Y \big(\Re (\omega_{t,0}+\i \chi)^m-\cot(\theta_0) \Im (\omega_{t,0}+\i \chi)^m \big) \\
& = & -\frac{1}{\sin(\theta_{0})}\int_{Y}\Im\left(e^{-\sqrt{-1}\theta_{0}}(\omega_{t,0}+\sqrt{-1}\chi)^{m}\right) \\[2mm]
& = & -\frac{1}{\sin(\theta_{0})}\int_{Y}\Im\left(e^{-\sqrt{-1}\theta_{0}}(\omega_{0}+\sqrt{-1}\chi+t\sigma)^{m}\right) \\
& = & -\frac{1}{\sin(\theta_{0})}\sum_{k=0}^{m}t^{m-k}\binom{m}{k}\int_{Y}\Im\left(e^{-\sqrt{-1}\theta_{0}}(\omega_{0}+\sqrt{-1}\chi)^{k}\right)\wedge\sigma^{m-k} \\
& = & \sum_{k=0}^{m}t^{m-k}\binom{m}{k}\int_{Y}\big(\Re (\omega_0+\i \chi)^k-\cot(\theta_0) \Im (\omega_0+\i \chi)^k \big)\wedge\sigma^{m-k}.
\end{eqnarray*}
The assumption shows each coefficient in this polynomial (of $t$) is positive when $m<n$ and non-negative when $m=n$. Then the triple $(X,\alpha,\beta)$ is stable along the test family $\omega_{t,0}$. This completes the proof of Corollary \ref{NM criterion C}.
\end{proof}

Next we show that Corollary \ref{NM criterion C} implies Corollary \ref{NM criterion D}.

\begin{proof}[Proof of Corollary \ref{NM criterion D}]
We follow the argument of \cite[Theorem 4.5]{DP04} (see also \cite[Theorem 1.1]{CT16}). Since $X$ is projective, then we choose $\gamma$ to be the first Chern class of a very ample line bundle $L$ over $X$. For any $m$-dimensional analytic subvariety $Y$ of $X$ ($m=1,\ldots,n$) and $1\leq k\leq m$, there exist generic members $H_{1},\ldots,H_{m-k}$ of the linear system $|L|$ such that $Y\cap H_{1}\cap\ldots\cap H_{m-k}$ is a $k$-dimensional analytic subvariety of $X$ and
\begin{eqnarray*}
&& \int_{Y}\big(\Re(\a+\i \b)^k-\cot(\theta_0) \Im(\a+\i \b)^k\big)\wedge\gamma^{m-k} \\
& = & \int_{Y\cap H_{1}\cap\ldots\cap H_{m-k}}\big(\Re(\a+\i \b)^k-\cot(\theta_0) \Im(\a+\i \b)^k\big).
\end{eqnarray*}
Then Corollary \ref{NM criterion D} follows from Corollary \ref{NM criterion C}.
\end{proof}

\section{The twisted dHYM equation on the resolution $\hat{Y}$} \label{The twisted dHYM equation on the resolution}
We will prove the remaining part (2) $\Rightarrow$ (3) of Theorem \ref{NM criterion A} by induction argument for $m:=\dim Y$. We will prove the following:
\begin{thm} \label{NM criterion B}
Let $X$ be an $n$-dimensional compact complex manifold, $\a$ a real $(1,1)$-cohomology class and $\b$ a K\"ahler class on $X$. Let $\theta_0 \in (0,\pi)$ be a constant satisfying
\begin{equation} \label{inequality for theta zero}
(\Re(\a+\i \b)^n-\cot(\theta_0) \Im(\a+\i \b)^n) \cdot X \geq 0,
\end{equation}
and assume the triple $(X,\a,\b)$ is stable along some test family $\omega_{t,0} \in \a_t$ ($t \in [0,\infty)$) emanating from $\omega_0 \in \a$. Then for any K\"ahler form $\chi \in \b$, we have
\begin{equation} \label{subsolution nonempty}
\Gamma_{\chi, \a, \theta_0, \theta_0}(Y,X) \neq \emptyset
\end{equation}
for all proper $m$-dimensional subvarieties $Y$ of $X$ ($m=1,\ldots,n-1$).
\end{thm}
In the above Theorem, we need the stronger result \eqref{subsolution nonempty} instead of $\Gamma_{\chi, \a, \theta_0, \pi}(Y,X) \neq \emptyset$ since unlike the function $Q_\chi$, the upper bound for $P_\chi$ is not preserved under the extension argument (see the proof of Theorem \ref{NM criterion B} in Section \ref{completion of the proof}). Also we remark that Theorem \ref{NM criterion B} is not true when $m=n$.\footnote{We take $\theta_0$ that makes \eqref{inequality for theta zero} an equality. If Theorem \ref{NM criterion B} is true for $Y=X$, then the condition \eqref{subsolution nonempty} contradicts the choice of $\theta_0$.} However, we can prove Theorem \ref{NM criterion A} by the same argument as \cite[Section 5]{Che21} as long as Theorem \ref{NM criterion B} holds (see Section \ref{completion of the proof} for more details).

So in the remaining part of the paper, we mainly focus on Theorem \ref{NM criterion B}. We will prove Theorem \ref{NM criterion B} by induction on dimension $m:=\dim Y$ for proper analytic subvarieties $Y \subset X$. The following lemma demonstrates the case when $m=1$ which initiates the induction argument.
\begin{lem}\label{0-dim and 1-dim}
Under the assumption of Theorem \ref{NM criterion B}, for any subvariety $Y$ of $X$ with $\dim Y \leq 1$, we have
\[
\Gamma_{\chi, \a, \theta_0, \theta_0}(Y,X) \neq \emptyset.
\]
\end{lem}

\begin{proof}
According to the value of $\dim Y$, there are two cases:

\vspace{4mm}

\noindent
\textbf{Case 1 ($\dim Y=0$)}: Lemma \ref{0-dim and 1-dim} is obvious since every cohomology class in a sufficiently small neighborhood of a point is trivial.

\vspace{4mm}

\noindent
\textbf{Case 2 ($\dim Y=1$)}: Since any K\"ahler classes on $Y$ are proportional to each other, by the condition $(\a-\cot(\theta_0) \b) \cdot Y>0$ and the extension theorem (\cf \cite[Proposition 3.3]{DP04}), we know that there exists a real $(1,1)$-form $\omega_Y \in \a|_{Y,X}$ on a neighborhood $U_Y$ of $Y$ such that
\[
\omega_Y-(\cot (\theta_0)+2\e)\chi>0
\]
for some small $\e>0$. Let $Y_{\sing}$ be the singular set of $Y$. We take $\varphi \in \PSH(U_Y \backslash Y_{\sing}, \e \chi) \cap C^\infty(U_Y \backslash Y_{\sing})$ such that $\varphi \to -\infty$ at $Y_{\sing}$. This implies that
\[
\omega_Y+\dd \varphi-(\cot (\theta_0)+\e)\chi>0.
\]
on $U_Y \backslash Y_{\sing}$. For simplicity, we assume that $Y_{\sing}$ is a point $p$. So by Case 1, there exists a neighborhood $U_p$ of $p$ and $\varphi_p \in C^\infty(U_p)$ such that
\[
\cQ_{\chi,n,n}(\omega_p)<\theta_0, \quad \omega_p:=\omega_Y+\dd \varphi_p
\]
on $U_p \Subset U_Y$. By subtracting a large constant from $\varphi_p$ we have $\varphi_p<\varphi-2$ on $U_Y \backslash V_p$ for some neighborhood $V_p \Subset U_p$ of $p$. Also since $\varphi$ has positive Lelong number along $Y_{\sing}$, we know that $\varphi (x) \to -\infty$ as $x \to p$. So there exists a neighborhood $W_p \Subset V_p$ of $p$ such that $\varphi_p>\varphi+2$ on $W_p$. Let $Y_1,\ldots,Y_\ell$ be the components of $Y$ in $U_Y \backslash W_p$ which are disjoint smooth open curves. Define a function
\[
\tilde{\varphi}:=\varphi+Ad^2
\]
for some constant $A>0$, where $d$ is the distance function to $Y_1,\ldots,Y_\ell$ with respect to any fixed K\"ahler metric on $Y$. So $\tilde{\omega}_Y:=\omega_Y+\dd \tilde{\varphi}$ produces large positive eigenvalues along normal direction of $\bigcup_{i=1}^\ell Y_i$ as $A$ increases. This implies that there exists a large constant $A$ such that
\[
\cQ_{\chi,n,n}(\tilde{\omega}_Y)<\theta_0
\]
on some sufficiently small open neighborhood $\tilde{U}$ of $Y \backslash W_p$ in $X$. After shrinking $\tilde{U}$ if necessary, we may assume that $\varphi_p<\tilde{\varphi}-1$ on $\tilde{U} \backslash V_p$ and $\varphi_p>\tilde{\varphi}+1$ on $\tilde{U} \cap W_p$. Then we may take the regularized maximum $\varphi_Y$ of $(\tilde{U}, \tilde{\varphi})$ and $(V_p, \varphi_p)$ (\cf \cite[Section 5.E]{Dem12})). The form $\omega_U:=\omega_Y+\dd \varphi_Y$ on a sufficiently small open neighborhood $U$ of $Y$ satisfies the desired condition as in Lemma \ref{0-dim and 1-dim}. Indeed, as pointed out in \cite[Section 4]{Che21}, the regularized maximum preserves the upper bound $\cQ_{\chi,n,n}(\omega_{U})<\theta_0$ due to the monotonicity and concavity properties (\cf Proposition \ref{monotone and concave matrix}).
\end{proof}

Using the induction argument, we assume that Theorem \ref{NM criterion B} is true for all quadruple $(Y,X,\a,\b)$ such that $(X,\a,\b)$ is stable along some test family and $\dim Y \leq m-1 \leq n-2$. Now let $X$ be an $n$-dimensional analytic variety satisfying the stability condition as in Theorem \ref{NM criterion B} for some test family $\omega_{t,0} \in \a_t$ ($t \in [0,\infty)$) emanating from $\omega_0 \in \a$, and $Y$ an $m$-dimensional analytic subvariety of $X$. So we have to show that
\[
\Gamma_{\chi, \a, \theta_0, \theta_0}(Y,X) \neq \emptyset.
\]
By Lemma \ref{0-dim and 1-dim}, we may assume that $n \geq 3$ and $2 \leq m \leq n-1$. We take the resolution of singularities $\Phi \colon X' \to X$ for $Y$ such that the strict transform $\hat{Y}$ of all components of $Y$ by $\Phi$ is a disjoint union of smooth submanifolds of $X'$. We assume that $Y$ is irreducible for simplicity in later arguments (in general case, we apply Theorem \ref{gluing of local subsolutions} to each component separately, and then prove Theorem \ref{NM criterion B} by the same argument). Let $Y_{\sing}$ be the singular set of $Y$ and set $E_0:=\hat{Y} \cap \Phi^{-1}(Y_{\sing})$ so that $\chi$ is non-degenerate on $\hat{Y} \backslash E_0$. We can choose $\Phi$ to be successive blow-ups along smooth centers. This implies that there exist a metric $h_{E_0}$ on the line bundle associated to $E_0$ and a small constant $\kappa_0>0$ such that $\chi-\kappa F_{h_{E_0}}$ is a K\"ahler form on $\hat{Y}$ for all $\kappa \in (0,\kappa_0]$, where $F_{h_{E_0}}$ denotes the curvature of $h_{E_0}$. In particular, we set
\[
\xi: = \chi-\kappa_{0}F_{h_{E_0}},
\]
and
\begin{equation}\label{def-phi}
\phi: = \kappa_{0}\log|\s_{E_0}|_{h_{E_0}}^2 \in \PSH(\hat{Y},\chi) \cap C^\infty(\hat{Y} \backslash E_0),
\end{equation}
where $\s_{E_0}$ is the defining section of the line bundle associated to $E_0$. Then the function $\phi$ has positive Lelong number along $E_0$ and satisfies
\[
\xi=\chi+\dd \phi
\]
on $\hat{Y} \backslash E_0$. By the definition of test families, changing variables if necessary, we may assume that
\begin{equation} \label{lower bound for the test family}
\omega_{1,0}-\cot \big( \frac{\theta_0}{n} \big)\chi \geq 0.
\end{equation}
For $\varrho \geq 0$, set
\[
\hat{\omega}_0=\hat{\omega}_0(t,\varrho):=\omega_{t,0}+\varrho t \xi, \quad \hat{\chi}=\hat{\chi}(t,\varrho):=\chi+(\varrho t)^n \xi.
\]
Then the associated cohomology class is given by
\[
\hat{\a}=\hat{\a}(t,\varrho):=\a_t+\varrho t[\xi], \quad \hat{\b}=\hat{\b}(t,\varrho):=\b+(\varrho t)^n [\xi]
\]
respectively.
\begin{lem} \label{stability on resolutions}
There exists a constant $\varrho_0 \in (0, \min \{ \big(\tan \big( \frac{\theta_0}{2n} \big) \big)^{1/(n-1)},1/100 \})$ and $c_{\hat{Y}}>0$ such that for any $t \in [0,1]$, $\varrho \in [0,\varrho_0]$ and $p$-dimensional subvariety $V$ of $\hat{Y}$ ($p=1,\ldots,m$), we have
\[
\big(\Re(\hat{\a}+\i \hat{\b})^m-\cot(\theta_0) \Im(\hat{\a}+\i \hat{\b})^m \big) \cdot \hat{Y} \geq c_{\hat{Y}}[\xi]^{m} \cdot \hat{Y}
\]
and
\[
\big(\Re(\hat{\a}+\i \hat{\b})^p-\cot(\theta_0) \Im(\hat{\a}+\i \hat{\b})^p \big) \cdot V \geq c_{\hat{Y}}(\varrho t)^{p}[\xi]^{m} \cdot V.
\]
\end{lem}
\begin{proof}
We may assume that $V$ is irreducible and set
\[
A_{\varrho,t}:=1+\i (\varrho t)^{n-1}.
\]

\noindent
\textbf{Case 1 ($p=m$)}: In this case, a direct computation shows that
\begin{eqnarray*}
&& \Im \big( e^{-\i \theta_0} (\hat{\a}+\i \hat{\b})^m \big) \cdot \hat{Y}\\[2.5mm]
&&= \Im \big( e^{-\i \theta_0} (\a_t+\i \b)^m \big) \cdot Y\\
&&+\Im \big( e^{-\i \theta_0} \sum_{i=0}^{m-1} (\varrho t)^{m-i} \dbinom{m}{i} (\a_t+\i \b)^i \cdot A_{\varrho,t}^{m-i} [\xi]^{m-i} \big) \cdot \hat{Y}.
\end{eqnarray*}
By the stability assumption and $m \leq n-1$, we know that the first term is negative for all $t \in [0,1]$. We can observe that for all $t \in [0,1]$, there exists $c>0$ such that $\Im(e^{-\i \theta_0} (\a_t+\i \hat{\b})^m) \cdot \hat{Y}<-c[\xi]^{m}\cdot\hat{Y}$ if $\varrho \geq 0$ is sufficiently small since the second term is of order $O(\varrho t)$ and $\a_t$ is fixed.

\vspace{4mm}

\noindent
\textbf{Case 2 ($p=1,\ldots,m-1$)}:
Assume $t \varrho>0$. Then a direct computation shows that
\begin{eqnarray*}
&&\Im \big( e^{-\i \theta_0}(\hat{\a}+\i \hat{\b})^p \big) \cdot V\\[2.5mm]
&&= \Im \big( e^{-\i \theta_0} (\a_t+\i \b)^p \big) \cdot V\\
&&+\sum_{i=1}^{p-1} (\varrho t)^{p-i} \dbinom{p}{i} \Im \big( e^{-\i \theta_0}(\a_t+\i \b)^i \cdot A_{\varrho,t}^{p-i} [\xi]^{p-i} \big) \cdot V\\
&&+(\varrho t)^p \Im \big( e^{-\i \theta_0} A_{\varrho,t}^p [\xi]^p \big) \cdot V\\
&&= \sum_{j=1}^5 I_j,
\end{eqnarray*}
where
\begin{eqnarray*}
I_1&:=&\Im \big( e^{-\i \theta_0} (\a_t+\i \b)^p \big) \cdot V,\\[2mm]
I_2&:=& \sum_{i=1}^{p-1} (\varrho t)^{p-i} \dbinom{p}{i} \Re (A_{\varrho,t}^{p-i}) \Im \big( e^{-\i \theta_0}(\a_t+\i \b)^i \big) \cdot [\xi]^{p-i} \cdot V,\\
I_3&:=& \sum_{i=1}^{p-1} (\varrho t)^{p-i} \dbinom{p}{i} \Im (A_{\varrho,t}^{p-i}) \Re \big(e^{-\i \theta_0}(\a_t+\i \b)^i \big) \cdot [\xi]^{p-i} \cdot V,\\[4mm]
I_4&:=& -(\varrho t)^p \Re(A_{\varrho,t}^p) \sin (\theta_0) [\xi]^p \cdot V,\\[7mm]
I_5&:=& (\varrho t)^p \Im(A_{\varrho,t}^p) \cos(\theta_0) [\xi]^p \cdot V.
\end{eqnarray*}
Let $W:=\Phi(V)$, so $W$ is a subvariety of $X$ of $\dim W \leq p \leq m-1$. By the induction hypothesis, there exists a real $(1,1)$-form $\omega_W \in \Gamma_{\chi, \a, \theta_0, \theta_0}(W,X)$. We consider the smooth family $\omega_{W,t}$ emanating from $\omega_W$:
\[
\omega_{W,t}:=\omega_{t,0}+\omega_W-\omega_0 \in \a_t.
\]
Since $\omega_{W,t} \geq \omega_W$, we have $\omega_{W,t} \in \Gamma_{\chi, \a_t, \theta_0, \theta_0}(W,X)$ for all $t \in (0,1]$, which shows that
\[
\Im (e^{-\i \theta_0} (\omega_{W,t}+\i \chi)^k)<0
\]
on a neighborhood of $W$ in $X$ for all $k=1,\ldots,n-1$ and $t \in (0,1]$ by Lemma \ref{positivity of forms}. Let $\omega_{V,t}:=\Phi^\ast \omega_{W,t}$. Then we have
\[
\Im (e^{-\i \theta_0} (\omega_{V,t}+\i \chi)^k) \leq 0,
\]
on a neighborhood of $V$ in $X'$ for all $k=1,\ldots,n-1$ and $t \in (0,1]$. In particular, we have $I_1 \leq 0$ for all $t \in (0,1]$. To compute $I_2$ and $I_4$, we observe that
\[
\Re(A_{\varrho,t}^i)>0, \quad \Im(A_{\varrho,t}^i)>0
\]
for all $i=1,\ldots,p$ and $t \in (0,1]$ by taking $\varrho>0$ sufficiently small so that $\varrho^{n-1}<\tan \big( \frac{\pi}{2n} \big)$. So we have $I_2 \leq 0$ and $I_4<0$.

To deal with bad terms $I_3$, $I_5$, we observe that exists a constant $C>0$ (depending only on $\theta_0$, $\{\omega_{t,0}\}$ and $\chi$) such that
\[
\Re \big( e^{-\i \theta_0}(\omega_{t,0}+\i \chi)^i \big) \leq C \xi^i
\]
on $\hat{Y}$ for all $i=1,\ldots,p-1$ and $t \in (0,1]$. This shows that
\[
I_3<I'_3:=C \sum_{i=1}^{p-1} (\varrho t)^{p-i} \dbinom{p}{i} \Im(A_{\varrho,t}^{p-i}) [\xi]^p \cdot V.
\]
We observe that
\[
\Im(A_{\varrho,t}^i) = (\varrho t)^{n-1}+O((\varrho t)^{n}).
\]
Then the coefficients of $[\xi]^p \cdot V$ in $I_3'$, $I_5$ are polynomials of $\varrho t$, which are of order $O((\varrho t)^n)$, $O((\varrho t)^{n+p-1})$ respectively. On the other hand, the coefficients of $[\xi]^p \cdot V$ in $I_4$ is a polynomials of $\varrho t$ of the form
\[
-\sin(\theta_0)(\varrho t)^p+O((\varrho t)^{2n+p-2}).
\]
Then there exists $c>0$ independent of $V$ such that $I'_3+I_4+I_5<-c(\varrho t)^{p}[\xi]^{m} \cdot V$ for all $t \in (0,1]$ and sufficiently small $\varrho>0$. This completes the proof.
\end{proof}
We will consider the following twisted dHYM equation for $\hat{\omega} \in \hat{\a}$ on $\hat{Y}$
\begin{equation} \label{twisted dHYM}
\Re(\hat{\omega}+\i \hat{\chi})^m-\cot(\theta_0) \Im(\hat{\omega}+\i \hat{\chi})^m-c_{t,\varrho} \hat{\chi}^m=0,
\end{equation}
where $c_{t,\varrho}$ is a normalized constant defined by
\begin{equation}\label{c t varrho}
c_{t,\varrho} \hat{\b}^m \cdot \hat{Y}=(\Re(\hat{\a}+\i \hat{\b})^m-\cot(\theta_0) \Im(\hat{\a}+\i \hat{\b})^m) \cdot \hat{Y}.
\end{equation}
The following lemma is crucial, and a direct consequence from Lemma \ref{stability on resolutions}.
\begin{lem} \label{uniform positivity for constant}
The constant $c_{t,\varrho}$ depends smoothly on $t$, $\varrho$, and is positive for all $t \in [0,1]$ and $\varrho \in [0,\varrho_0]$.
\end{lem}
We take a constant $\Theta_0 \in (\theta_0, \pi)$ (depending only on $n$ and $\theta_0$) sufficiently close to $\theta_0$ so that
\[
\Theta_0-\theta_0<\frac{\pi-\Theta_0}{n}.
\]
For any fixed $\varrho \in (0,\varrho_0]$, we set
\[
\cT_\varrho:=\{t \in (0,1]|\text{\eqref{twisted dHYM} has a smooth solution $\hat{\omega}(t,\varrho) \in \Gamma_{\hat{\chi}, \hat{\a}, \theta_0, \Theta_0}(\hat{Y})$} \}.
\]
We remark that for any $t \in \cT_\varrho$, the solution $\hat{\omega} \in \Gamma_{\hat{\chi}, \hat{\a}, \theta_0, \Theta_0}(\hat{Y})$ to \eqref{twisted dHYM} automatically satisfies
\begin{equation} \label{upper bound for Q}
P_{\hat{\chi}}(\hat{\omega})<Q_{\hat{\chi}}(\hat{\omega})<\theta_0
\end{equation}
since $Q_{\hat{\chi}}(\hat{\omega})<\Theta_0<\pi$ and $c_{t,\varrho}>0$.
\begin{lem} \label{solvability at one}
For any fixed $\varrho \in (0,\varrho_0]$, the set $\cT_\varrho$ is open in $(0,1]$. Moreover, we have $1 \in \cT_\varrho$.
\end{lem}
\begin{proof}
Openness follows from Proposition \ref{continuity method} and the fact that $c_{t,\varrho}>0$ for all $t \in (0,1]$. When $t=1$, by \eqref{lower bound for the test family}, we know that
\[
\hat{\omega}_0 \geq \cot \big( \frac{\theta_0}{n} \big) \chi+\varrho \xi, \quad \hat{\chi}=\chi+\varrho^n \xi,
\]
and
\[
P_{\hat{\chi}}(\hat{\omega}_0)<Q_{\hat{\chi}}(\hat{\omega}_0) \leq Q_{\hat{\chi}} \big( \cot \big( \frac{\theta_0}{n} \big) \chi+\varrho \xi \big).
\]
Let $\mu_i$ be the eigenvalues of $\chi$ with respect to $\xi$. Then
\[
Q_{\hat{\chi}}\big( \cot \big( \frac{\theta_0}{n} \big) \chi+\varrho \xi \big)=\sum_{i=1}^m \arccot \frac{\cot \big( \frac{\theta_0}{n} \big) \mu_i+\varrho}{\mu_i+\varrho^n}<\sum_{i=1}^m \arccot \bigg( \cot \big( \frac{\theta_0}{n} \big) \bigg)<\theta_0
\]
since $\varrho^{n-1} \leq \varrho_0^{n-1}<\tan \big( \frac{\theta_0}{n} \big)$. Thus we have $1 \in \cT_\varrho$ by Proposition \ref{continuity method}.
\end{proof}
Set $t_\varrho:=\inf \cT_\varrho$.
\begin{lem} \label{solvability for twisted dHYM equation}
For any fixed $\varrho \in (0,\varrho_0]$, we have $t_\varrho=0$.
\end{lem}
\begin{proof}
Assume $t_\varrho>0$. Take any element $\eta \in \hat{\a}(t_\varrho)$ and consider the test family emanating from $\eta$
\[
\zeta_{s,t_\varrho}:=\eta+s \xi, \quad s \in [0,\infty)
\]
with respect to the background metric $\hat{\chi}(t_\varrho)$. There exists a constant $C>0$ such that $\hat{\chi}(t_\varrho)\leq C\xi$. For any $p$-dimensional analytic subvariety $V$ of $\hat{Y}$ ($p=1,\ldots,m$), let us consider the function $F_{\theta_0}^{\Stab}(V, \{\zeta_{s,t_\varrho} \},s)$. By Lemma \ref{stability on resolutions}, we have
\[
F_{\theta_0}^{\Stab}(V, \{\zeta_{s,t_\varrho} \},0) \geq c_{\hat{Y}}(\varrho t_{\varrho})^{p}[\xi]^{p}\cdot V \geq c_{\hat{Y}}(C^{-1}\varrho t_{\varrho})^{p}[\hat{\chi}(t_\varrho)]^{p}\cdot V
\]
By the definition of $\cT_\varrho$, for any $t \in (t_\varrho, 1]$ we have a solution $\hat{\omega}(t) \in \Gamma_{\hat{\chi}(t), \hat{\a}(t), \theta_0, \Theta_0}(\hat{Y})$ to \eqref{twisted dHYM}. In order to compute $F_{\theta_0}^{\Stab}(V, \{\zeta_{s,t_\varrho} \},s)$, for any $t \in (t_\varrho,1]$, let us consider the test family
\[
\zeta_{s,t}:=\hat{\omega}(t)+s \xi, \quad s \in [0,\infty)
\]
emanating from $\hat{\omega}(t)$ with the background metric $\hat{\chi}(t)$. Since $\hat{\omega}(t) \in \Gamma_{\hat{\chi}(t), \hat{\a}(t), \theta_0, \Theta_0}(\hat{Y})$, we know that $\zeta_{s,t} \in \Gamma_{\hat{\chi}(t), \hat{\a}(t), \theta_0, \Theta_0}(\hat{Y})$ by the monotonicity of $P_{\hat{\chi}(t)}$, $Q_{\hat{\chi}(t)}$, and hence $\frac{d}{ds} F_{\theta_0}^{\Stab}(V, \{\zeta_{s,t}\},s) \geq 0$ for all $s \in [0,\infty)$. Moreover, since $[\zeta_{s, t}] \to [\zeta_{s,t_\varrho}]$ as $t \to t_\varrho$ and $\frac{d}{ds} [\zeta_{s, t}]=\frac{d}{ds} [\zeta_{s,t_\varrho}]=[\xi]$, we have
\[
\frac{d}{ds} F_{\theta_0}^{\Stab}(V, \{\zeta_{s,t_\varrho} \},s)=\lim_{t \to t_\varrho} \frac{d}{ds} F_{\theta_0}^{\Stab}(V, \{\zeta_{s,t}\},s) \geq 0
\]
for all $s \in [0,\infty)$, where we used the fact that $\frac{d}{ds} F_{\theta_0}^{\Stab}(V, \{\zeta_{s,t}\},s)$ is a cohomological invariant. This shows that
\[
F_{\theta_0}^{\Stab}(V, \{\zeta_{s,t_\varrho} \},s) \geq
F_{\theta_0}^{\Stab}(V, \{\zeta_{s,t_\varrho} \},0) \geq c_{\hat{Y}}(C^{-1}\varrho t_{\varrho})^{p}[\hat{\chi}(t_\varrho)]^{p}\cdot V
\]
for all $s \in [0,\infty)$, so the triple $(\hat{Y},\hat{\a}({t_\varrho}), \hat{\b} ({t_\varrho}))$ together with the test family $\zeta_{s,t_\varrho}$ satisfy the uniform stable assumption of \cite[Proposition 5.2]{Che21}. Then there exists a solution $\hat{\omega}(t_\varrho) \in \Gamma_{\hat{\chi}(t_\varrho),\hat{\a}(t_\varrho),\theta_0,\Theta_0}(\hat{Y})$ to \eqref{twisted dHYM}. This is a contradiction.
\end{proof}

\section{The twisted dHYM equation on the product space $\cY$} \label{the twisted dHYM equation on the product space}
In this section, we consider the twisted dHYM equation on the product space $\cY:=\hat{Y} \times \hat{Y}$ as in \cite{Che21,DP04,Son20}. Set $\Delta:=\{(p,p) \in \cY|p \in \hat{Y}\}$ and suppose that $\Delta$ is locally defined by holomorphic functions $\{f_{j,k}\}_{j,k}$ on finitely many domains $\{\cU_j\}$ covering $\cY$. We define
\[
\psi_s:=\log \bigg( \sum_j \r_j \sum_k |f_{j,k}|^2+s^2 \bigg)
\]
for $s \in (0,1]$ where $\{\r_j\}$ is a partition of unity for $\{\cU_j\}$. Let $\pi_1, \pi_2$ be the projection maps from $\cY \to \hat{Y}$ and set
\begin{equation} \label{definition of chi y}
\chi_\cY:=\pi_1^\ast \hat{\chi}+\pi_2^\ast \hat{\chi},
\end{equation}
\begin{equation} \label{definition of chi ys}
\chi_{\cY,s}:=\chi_\cY-\ell \kappa_{0}(\pi_1^\ast F_{h_{E_0}}+\pi_2^\ast F_{h_{E_0}})+\ell^2 \dd \psi_s,
\end{equation}
where the function $\phi \in \PSH(\hat{Y},\chi) \cap C^\infty(\hat{Y} \backslash E_0)$ is defined in Section \ref{The twisted dHYM equation on the resolution}, and the small constant $\ell>0$ is determined later. If we set
\begin{equation}\label{c t varrho ell}
\int_\cY \chi_{\cY, s}^{2m}=(1+c_{t,\varrho,\ell}) \int_\cY \chi_\cY^{2m},
\end{equation}
then we have $\lim_{\ell \to 0} c_{t,\varrho,\ell}=0$ uniformly for $t \in [0,1]$ and $\varrho \in [0,\varrho_0]$ (we note that the constant $c_{t,\varrho,\ell}$ is independent of $s \in (0,1]$). We fix a constant $K>0$ (depending only on $m$ and $\theta_0$) such that
\[
K>\cot \big( \frac{\pi-\theta_0}{n} \big),
\]
and define constants $\zeta_K, \tilde{\theta}_0$ by
\[
\zeta_K:=m \arccot(K), \quad 0<\tilde{\theta}_0:=\theta_0+\zeta_K<\pi.
\]
Now let us consider the twisted dHYM equation for $\omega_{\cY,s} \in \pi_1^\ast \hat{\a}+K\pi_2^\ast \hat{\b}$ on $\cY$
\begin{equation} \label{twisted dHYM b}
\Re(\omega_{\cY,s}+\i \chi_\cY)^{2m}-\cot(\tilde{\theta}_0)\Im (\omega_{\cY,s}+\i \chi_\cY)^{2m}-F_{\cY,s} \chi_\cY^{2m}=0,
\end{equation}
where the smooth function $F_{\cY,s}$ is given by
\begin{equation} \label{definition of F ys}
F_{\cY,s}:=\frac{\chi_{\cY,s}^{2m}}{\chi_\cY^{2m}}+c_{t,\varrho} \frac{\sin(\theta_0)}{\sin (\tilde{\theta}_0)}(K^2+1)^{m/2}-(1+c_{t,\varrho,\ell}).
\end{equation}
Also we set
\[
R_K:=\Re(K+\i)^m, \quad I_K:=\Im (K+\i)^m>0.
\]
Then we have $I_K>0$ since $0<\zeta_K<\pi$. Also we note that $R_K/I_K=\cot(\zeta_K)$.
\begin{lem} \label{solvability of the twisted dHYM}
There exists $\ell>0$ such that the equation \eqref{twisted dHYM b} has a solution $\omega_{\cY,s} \in \Gamma_{\chi_\cY, \pi_1^\ast \hat{\a}+K\pi_2^\ast \hat{\b}, \tilde{\theta}_0, \tilde{\theta}_0}(\cY)$ for all $t \in (0,1]$, $\varrho \in (0,\varrho_0]$ and $s \in (0,1]$.
\end{lem}
\begin{proof}
In order to prove the existence of the solution, we may check the conditions imposed in Proposition \ref{continuity method}. First, we show the existence of subsolution. Set $\omega_\cY^\dag:=\pi_1^\ast \hat{\omega}+K \pi_2^\ast \hat{\chi}$, where $\hat{\omega}$ is a solution to \eqref{twisted dHYM}. Let $\lambda_i$ be the eigenvalues of $\hat{\omega}$ with respect to $\hat{\chi}$. Then the eigenvalues of $\omega_\cY^\dag$ with respect to $\chi_\cY$ are given by
\[
(\underbrace{\lambda_1,\ldots,\lambda_m}_{m},\underbrace{K,\ldots,K}_{m}).
\]
Combining with \eqref{upper bound for Q}, we have
\[
Q_{\chi_\cY}(\omega_\cY^\dag)=Q_{\hat{\chi}}(\hat{\omega})+m \arccot(K)<\theta_0+m \arccot(K)=\tilde{\theta}_0.
\]
Next, we check the integrals of both sides of \eqref{twisted dHYM b} on $\cY$ are equal:
\[
\int_{\cY}F_{\cY,s} \chi_\cY^{2m} =
\left( \Re(\omega_{\cY,s}+\i \chi_\cY)^{2m}-\cot(\tilde{\theta}_0)\Im (\omega_{\cY,s}+\i \chi_\cY)^{2m} \right)\cdot\cY.
\]
Using \eqref{c t varrho ell} and \eqref{c t varrho}, we compute
\[
\begin{split}
& \int_{\cY}F_{\cY,s} \chi_\cY^{2m} \\
= {} & c_{t,\varrho} \frac{\sin(\theta_0)}{\sin (\tilde{\theta}_0)}(K^2+1)^{m/2} \cdot \binom{2m}{m} \cdot \left(\hat{\beta}^{m}\cdot\hat{Y}\right)^2\\
= {} & -\frac{(K^2+1)^{m/2}}{\sin (\tilde{\theta}_0)} \cdot \binom{2m}{m} \cdot \left(\hat{\beta}^{m}\cdot\hat{Y}\right)
\int_{\hat{Y}}\Im\left(e^{-\sqrt{-1}\theta_{0}}(\hat{\alpha}+\sqrt{-1}\hat{\beta})^{m}\right)\\
= {} & -\frac{(K^2+1)^{m/2}}{\sin (\tilde{\theta}_0)} \cdot \binom{2m}{m} \cdot \left\{
\Im\int_{\cY} e^{-\sqrt{-1}\tilde{\theta}_{0}}\pi_{2}^{*}(e^{\sqrt{-1}\arccot(K)}\hat{\beta})^{m}\wedge \pi_{1}^{*}(\hat{\alpha}+\sqrt{-1}\hat{\beta})^{m}\right\}\\
= {} & -\frac{1}{\sin (\tilde{\theta}_0)}\left\{
\Im\int_{\cY} e^{-\sqrt{-1}\tilde{\theta}_{0}}\left[(\pi_{1}^{*}\hat{\alpha}+K\pi_{2}^{*}\hat{\beta})
+\sqrt{-1}(\pi_{1}^{*}\hat{\beta}+\pi_{2}^{*}\hat{\beta})\right]^{2m}\right\}\\
= {} & \left( \Re(\omega_{\cY,s}+\i \chi_\cY)^{2m}-\cot(\tilde{\theta}_0)\Im (\omega_{\cY,s}+\i \chi_\cY)^{2m} \right)\cdot\cY.
\end{split}
\]

Last, we check $\inf_\cY F_{\cY,s}>0$. From the argument in \cite[Lemma 2.1 (ii)]{DP04}, there exists $A>1$ such that for any $s \in (0,1]$, we have
\[
 A (\pi_1^\ast \xi+\pi_2^\ast \xi)+\dd \psi_s>0.
\]
This shows that for any $\e>0$,
\[
\chi_{\cY,s}-(1-\e)\chi_\cY \geq \pi_1^\ast \g_\ell+\pi_2^\ast \g_\ell,
\]
where
\begin{eqnarray*}
\g_\ell &:=& \e \chi-\ell \kappa_{0}F_{h_{E_0}}-A\ell^2 \xi\\
&=& (\e-A\ell^2) \bigg( \chi-\frac{\ell\kappa_{0}(1-A\ell)}{\e-A \ell^2} F_{h_{E_0}} \bigg),
\end{eqnarray*}
which shows that if $\ell<\min \{1/2A, \sqrt{\e/2A}, \e \}$, we have $\g_\ell>0$, and hence
\[
\chi_{\cY,s} \geq (1-\e)\chi_\cY > 0.
\]
This implies that
\begin{equation} \label{crucial inequality for F 1}
\begin{split}
F_{\cY,s}
= {} & \frac{\chi_{\cY,s}^{2m}}{\chi_\cY^{2m}}+c_{t,\varrho} \frac{\sin(\theta_0)}{\sin (\tilde{\theta}_0)}(K^2+1)^{m/2}-(1+c_{t,\varrho,\ell}) \\
\geq {} & \e\frac{\chi_{\cY,s}^{2m}}{\chi_\cY^{2m}}+(1-\e)^{2m+1}+c_{t,\varrho} \frac{\sin(\theta_0)}{\sin (\tilde{\theta}_0)}(K^2+1)^{m/2}-(1+c_{t,\varrho,\ell}).
\end{split}
\end{equation}
We find that the third term of the RHS has uniform positive lower bound for $t \in (0,1]$ and $\varrho \in (0,\varrho_0]$ by Lemma \ref{uniform positivity for constant}. Combining with the fact that $\lim_{\ell \to 0} c_{t,\varrho,\ell}=0$, we can take $\e>0$ sufficiently small so that for all $s \in (0,1]$, $t \in (0,1]$ and $\varrho \in (0,\varrho_0]$,
\begin{equation} \label{crucial inequality for F 2}
F_{\cY,s} \geq \e\frac{\chi_{\cY,s}^{2m}}{\chi_\cY^{2m}} > 0.
\end{equation}
So we can apply Proposition \ref{continuity method} to get the solution $\omega_{\cY,s}$, which automatically satisfies
\[
P_{\chi_\cY}(\omega_{\cY,s})<Q_{\chi_\cY}(\omega_{\cY,s})<\tilde{\theta}_0
\]
since $\inf_\cY F_{\cY,s}>0$. This completes the proof.
\end{proof}
\begin{rk} \label{twisting function in the top dimensional case}
In the above proof, the crucial point is that the third term of the RHS of \eqref{crucial inequality for F 1} has a uniform positive lower bound, which is not true for the case $m=n$ since the constant $c_{t,\varrho}$ may be equal to zero when $t=\varrho=0$.
\end{rk}
In later arguments, we fix $\e>0$ and $\ell>0$ in the proof of Lemma \ref{solvability of the twisted dHYM}. In the course of the proof, we also find that
\[
\chi_{\cY,s} \geq \pi_1^\ast \g_\ell+\pi_2^\ast \g_\ell>0,
\]
so $\chi_{\cY,s}$ is K\"ahler for any $s \in (0,1]$, $t \in [0,1]$ and $\varrho \in [0,\varrho_0]$.

\section{The mass concentration and local smoothing} \label{the mass concentration and local smoothing}
In this section, we will construct a subsolution as a current, and establish some estimates for its local smoothing. Compared to \cite[Section 5]{Son20}, we need to make some modifications to deal with the problem that the solution $\omega_{\cY,s}$ to \eqref{twisted dHYM b} is no longer K\"ahler. As pointed out in \cite[Section 5]{Che21}, we can overcome this by using a uniform lower bound for $\omega_{\cY,s}$ which we will discuss at first. Since $\omega_{\cY,s} \in \Gamma_{\chi_\cY, \pi_1^\ast \hat{\a}+K \pi_2^\ast \hat{\b}, \tilde{\theta}_0, \tilde{\theta}_0}(\cY)$, we get a lower bound of a real $(1,1)$-form $\Psi_{\cY,s}$ defined by
\begin{equation} \label{lower bound for Phi}
\Psi_{\cY,s}:=\omega_{\cY,s}+C_{\theta_0} \chi_\cY>c_{\theta_0} \chi_\cY
\end{equation}
for some fixed constants $C_{\theta_0}, c_{\theta_0}>0$ depending only on $\theta_0$. Thus $\Psi_{\cY,s}$ is K\"ahler and after passing to a sequence $\{s_i\}$, we know that for all $k=1,\ldots,m$, the $(k,k)$-form $\Psi_{\cY,s_i}^k$ converges to some closed positive $(k,k)$-current $\Xi_k=\Xi_k(t,\varrho)$ weakly when $s_i \to 0$ and $t$, $\varrho$ are fixed. For any $\eta>0$, let $\Delta_\eta$ be the $\eta$-neighborhood of the diagonal set $\Delta$ with respect to any fixed K\"ahler metric on $\cY$. We set
\[
\cV:=\Im(K+\i)^m \hat{\b}^m \cdot \hat{Y} \geq \Im(K+\i)^m \b^m \cdot Y>0.
\]
\begin{lem} \label{mass concentration}
We have the following:
\begin{enumerate}
\item For $k=1,\ldots, m-1$ and any positive $(m-k, m-k)$-form $\g$, the $(m,m)$-current $\Xi_k \wedge \g$ has no concentration of mass on $\Delta$, \ie for any $\e>0$, we have
\[
\int_{\Delta_\eta} \Xi_k \wedge \gamma<\e
\]
for sufficiently small $\eta>0$.
\item There exists $\d_0 \in (0,1)$ (independent of $t$, $\varrho$ and the choice of $\{s_i\}$) such that for any $t \in (0,1]$ and $\varrho \in (0,\varrho_0]$, we have $\Xi_m \geq 100 \cV \d_0 [\Delta]$.
\end{enumerate}
\end{lem}
\begin{proof}
By \cite[Corollary III. 2.11]{Dem12}, we have $1_{\Delta}\Xi_k=0$ and so $1_{\Delta}(\Xi_k \wedge \gamma)=0$, which implies (1).

As for (2), let us recall the twisted dHYM equation \eqref{twisted dHYM b}:
\[
\Re(\omega_{\cY,s}+\i \chi_\cY)^{2m}-\cot(\tilde{\theta}_0)\Im (\omega_{\cY,s}+\i \chi_\cY)^{2m}-F_{\cY,s} \chi_\cY^{2m}=0.
\]
Combining this with \eqref{lower bound for Phi} and \eqref{crucial inequality for F 2}, we obtain
\begin{equation} \label{lower volume bound}
\Psi_{\cY,s}^{2m} \geq c_0 F_{\cY,s} \chi_\cY^{2m} \geq c_0 \e \big( \chi_{\cY,s}(t,\varrho) \big)^{2m} \geq c_0 \ve \big( \chi_{\cY,s}(0,0) \big)^{2m},
\end{equation}
where $c_0>0$ is a constant depending only on $\theta_0$ and $\e$ is the fixed constant in the end of Section \ref{the twisted dHYM equation on the product space}. On the other hand, we know that $1_\Delta \Xi_m$ is a closed positive current with support in $\Delta$ by Skoda's extension theorem. Hence $1_\Delta \Xi_m=\d' [\Delta]$ for some $\d' \geq 0$ by the support theorem. In the same way as \cite[Proposition 2.6]{DP04} (or the arguments in \cite[Section 3]{Che21} for the $J$-equation case), one can see that there exists $\d_0 \in (0,1)$ such that $\Xi_m \geq 1_\Delta \Xi_m=100 \cV \d_0 [\Delta]$. This completes the proof.
\end{proof}
For $k=1,\ldots,m$, we set
\[
\omega_{\cY,k,s}^+:=\sum_{\text{$i$:even}} \dbinom{k}{i} C_{\theta_0}^i \Psi_{\cY,s}^{k-i} \wedge \chi_\cY^i, \quad \omega_{\cY,k,s}^-:=-\sum_{\text{$i$:odd}} \dbinom{k}{i} C_{\theta_0}^i \Psi_{\cY,s}^{k-i} \wedge \chi_\cY^i,
\]
Then we know that $\pm \omega_{\cY,k,s}^{\pm}$ defines a closed positive $(k,k)$-form with a decomposition
\[
\omega_{\cY,s}^k = (\Psi_{\cY,s}-C_{\theta_0} \chi_\cY)^{k}
= \omega_{\cY,k,s}^+ + \omega_{\cY,k,s}^-, \quad \omega_{\cY,k,s}^+ \geq \Psi_{\cY,s}^k.
\]
The form $\omega_{\cY,k,s_i}^+$, $\omega_{\cY,k,s_i}^-$ converges to
\[
\Upsilon_k^+=\Upsilon_k^+(t,\varrho):=\sum_{\text{$i$:even}} \dbinom{k}{i} C_{\theta_0}^i \Xi_{k-i} \wedge \chi_\cY^i, \quad \Upsilon_k^-=\Upsilon_k^-(t,\varrho):=-\sum_{\text{$i$:odd}} \dbinom{k}{i} C_{\theta_0}^i \Xi_{k-i} \wedge \chi_\cY^i.
\]
as $s_i \to 0$ respectively. Again $\pm \Upsilon_k^{\pm}$ are closed positive $(k,k)$-currents. We remark that by Lemma \ref{mass concentration}, for any fixed positive $(m-k,m-k)$-form $\g$, $\Upsilon_k^+ \wedge \g$ ($k=1,\ldots,m-1$) and $-\Upsilon_k^- \wedge \g$ ($k=1,\ldots,m$) have no concentration of mass on $\Delta$, whereas the current $\Upsilon_m^+$ satisfies $\Upsilon_m^+ \geq \Xi_m \geq 100 \cV \d_0 [\Delta]$.

Now we recall the argument in \cite[Section 5]{Che21}. We define a real closed $(1,1)$-form $\omega_s$ by the fiber integration:
\[
\omega_s:=\omega_s(t,\varrho)=\cV^{-1} \sum_{k=0}^{\lfloor \frac{m-1}{2} \rfloor}(-1)^k \frac{m!}{(m-2k)!(2k+1)!}(\pi_1)_\ast(\omega_{\cY,s}^{m-2k} \wedge \pi_2^\ast \hat{\chi}^{2k+1}).
\]
Then one can check that $\omega_s \in \hat{\a}$ as follows:
\[
\begin{split}
[\omega_{s}]&=\mathcal{V}^{-1} \sum_{k=0}^{\lfloor{\frac{m-1}{2}}\rfloor}(-1)^k\frac{m!}{(m-2k)!(2k+1)!}
(\pi_1)_*\left([\omega_{\cY,s}]^{m-2k}\wedge [\pi_{2}^{*}\hat{\chi}]^{2k+1}\right)\\
&=\mathcal{V}^{-1} \sum_{k=0}^{\lfloor{\frac{m-1}{2}}\rfloor}(-1)^k\frac{1}{m-2k}\binom{m}{2k+1}
(\pi_1)_*\left([\pi_1^*\hat\omega+K\pi_2^*\hat\chi]^{m-2k}\wedge [\pi_2^*\hat \chi]^{2k+1}\right)\\
&=\mathcal{V}^{-1} \sum_{k=0}^{\lfloor{\frac{m-1}{2}}\rfloor}(-1)^k\binom{m}{2k+1}
(\pi_1)_*\left([\pi_1^*\hat\omega]\wedge [K\pi_2^*\hat\chi]^{m-2k-1}\wedge [\pi_2^*\hat \chi]^{2k+1}\right)\\[3mm]
&=\mathcal{V}^{-1} (\pi_1)_*\left\{ [\pi_1^*\hat \omega]\wedge [\pi_2^*\Im(K\hat \chi+\sqrt{-1}\hat\chi)^m] \right\}
=[\hat \omega].
\end{split}
\]
Now we will compute $Q_{\hat{\chi}}(\omega_s)$. We write
\[
\omega_{\cY,s}=\omega_H+\omega_M+\omega_{\overline{M}^T}+\omega_V,
\]
where $\omega_H$ is the horizontal component, $\omega_V$ is the vertical component, $\omega_M$ and $\omega_{\overline{M}^T}$ are the off-diagonal or the mixed components of $\omega_{\cY,s}$ respectively. At any point $p \in \pi_1^{-1}(z)$ for $z \in \hat{Y}$, we can choose a coordinate so that
\[
\chi_H=\pi_1^\ast \hat{\chi}=\i \sum_{i=1}^m dz_i \wedge dz_{\bar{i}}, \quad \chi_V=\pi_2^\ast \hat{\chi}=\i \sum_{i=1}^m dw_i \wedge dw_{\bar{i}},
\]
\[
\omega_H=\i \sum_{i=1}^m \mu_i dz_i \wedge dz_{\bar{i}}, \quad \omega_V=\i \sum_{i=1}^m \lambda_i dw_i \wedge dw_{\bar{i}},
\]
\[
\omega_M=\sum_{i,j=1}^m \i \omega_{i \bar{j}}^M dz_i \wedge dw_{\bar{j}}, \quad
\omega_{\overline{M}^T}=\overline{\omega_M}.
\]
Then the calculation of \cite[Proof of Proposition 5.13]{Che21} shows that $\omega_s$ can be expressed as
\begin{equation} \label{fiber integral}
\omega_s=\cV^{-1} (\pi_1)_\ast \bigg( \tilde{\omega}_{\cY,s} \wedge \Im(\omega_V+\i \chi_V)^m \bigg),
\end{equation}
where in this coordinate, the form $\tilde{\omega}_{\cY,s}$ is given by
\[
\tilde{\omega}_{\cY,s}:=\sum_{i,j,l=1}^m \bigg( \mu_i \d_{i \bar{j}}-\omega_{i \bar{l}}^M \frac{\Im \prod_{k=1, k \neq l}^m(\lambda_k+\i)}{\Im \prod_{k=1}^m (\lambda_k+\i)} \overline{\omega_{j \bar{l}}^M} \bigg) dz_i \wedge dz_{\bar{j}}.
\]
By Lemma \ref{solvability of the twisted dHYM} and \cite[Lemma 5.12]{Che21}, we know that for all $t \in (0,1]$, $\varrho \in (0,\varrho_0]$ and $s \in (0,1]$, we have
\begin{equation} \label{estimates Q product}
Q_{\chi_H}(\tilde{\omega}_{\cY,s})+Q_{\chi_V}(\omega_V) \leq \tilde{\theta}_0,
\end{equation}
where $Q_{\chi_H}(\tilde{\omega}_{\cY,s})$ is taken for the horizontal component of $\tilde{\omega}_{\cY,s}$. In particular, this shows that the $(m,m)$-form $\Im(\omega_V+\i \chi_V)^m$ defines a positive measure when restricted to each fiber.
\begin{lem} \label{estimates for omega s}
For any $t \in (0,1]$, $\varrho \in (0,\varrho_0]$ and $s \in (0,1]$, we have
\[
Q_{\hat{\chi}}(\omega_s) \leq \theta_0.
\]
\end{lem}
\begin{proof}
The proof is identical to \cite[Proposition 5.13]{Che21}. By using the monotonicity and concavity property for $\cot(Q_{\hat{\chi}}(\cdot))$ with \eqref{estimates Q product}, for any $z \in \hat{Y}$ we observe that
\begin{eqnarray*}
\frac{1}{\cot(Q_{\hat{\chi}}(\omega_s(z)))-\cot(\tilde{\theta}_0)} &\leq& \cV^{-1} \int_{\pi_1^{-1}(z)} \frac{\Im(\omega_V+\i \chi_V)^m}{\cot(Q_{\chi_H}(\tilde{\omega}_{\cY,s}))-\cot(\tilde{\theta}_0)}\\
&\leq& \cV^{-1} \int_{\pi_1^{-1}(z)} \frac{\Im(\omega_V+\i \chi_V)^m}{\cot(\tilde{\theta}_0-Q_{\chi_V}(\omega_V))-\cot(\tilde{\theta}_0)}\\
&=& \cV^{-1} \int_{\pi_1^{-1}(z)} \frac{\cot(Q_{\chi_V}(\omega_V))-\cot(\tilde{\theta}_0)}{1+\cot^2(\tilde{\theta}_0)} \Im(\omega_V+\i \chi_V)^m\\
&=& \cV^{-1} \int_{\pi_1^{-1}(z)} \frac{\Re (\omega_V+\i \chi_V)^m-\cot(\tilde{\theta}_0) \Im (\omega_V+\i \chi_V)^m}{1+\cot^2(\tilde{\theta}_0)}\\
&=& \cV^{-1} \int_{\hat{Y}} \frac{\Re (K\hat{\chi}+\i \hat{\chi})^m-\cot(\tilde{\theta}_0) \Im (K\hat{\chi}+\i\hat{\chi})^m}{1+\cot^2(\tilde{\theta}_0)}\\
&=&  \frac{R_{K}-\cot(\tilde{\theta}_0)I_{K}}{(1+\cot^2(\tilde{\theta}_0))I_{K}} = \frac{1}{\cot(\theta_0)-\cot(\tilde{\theta}_0)},
\end{eqnarray*}
where we used $R_K/I_K=\cot(\zeta_K)$ and $\tilde{\theta}_0=\theta_0+\zeta_K$ in the last equality.
\end{proof}
Now let us introduce the local smoothing on the manifold $\hat{Y}$.
\begin{dfn} \label{local smoothing}
Let $T$ be a locally $L^1$-integrable function or closed $(1,1)$-current on $\hat{Y}$, and $R>0$ a constant. Then the {\it local smoothing} $T^{(r)}=\{T_j^{(r)}\}_{j \in \cJ}$ of $T$ with respect to a finte covering $\{B_{j,3R}\}_{j \in \cJ}$ of $\hat{Y}$ and a scale $r \in (0,R)$ is defined as follows:
\begin{enumerate}
\item For each $j \in \cJ$, $B_{j,4R}$ is a coordinate chart of $\hat{Y}$ which is biholomorphic to a Euclidean ball $B_{4R}(0)$ of radius $4R$ centered at the origin in $\C^m$ with respect to the standard Euclidean metric $g_j$. We assume that $B_{j,R} \simeq B_R(0) \subset \C^m$ ($j \in \cJ$) is also a covering of $\hat{Y}$.

\item $T^{(r)}_j(z)$ is the smoothing on $B_{j,3R}$ defined by the following convolution
\[
T_j^{(r)}(z):=\int_{B_r(0)} r^{-2m} \r \bigg( \frac{|y|}{r} \bigg) T(z-y) d \Vol_{g_j}(y).
\]
for $r \in (0,R)$, where $\r(t)$ is a smooth non-negative function with support in $[0,1]$, and is constant on $[0,1/2]$. Moreover, the function $\r$ satisfies the normalization condition
\[
\int_{B_1(0)} \r(|y|) d \Vol_{g_j}(y)=1.
\]
\end{enumerate}
\end{dfn}
By taking $R>0$ sufficiently small, we can always assume that
\begin{equation} \label{equivalence to Euclidean metrics}
\frac{1}{2} g_j \leq \xi \leq 2 g_j
\end{equation}
on each $B_{j,4R}$. Also for any small $\hat{\epsilon} \in (0,1]$, we can choose a sufficiently small $R>0$ and a finite covering $\{B_{j,3R}\}_{j \in \cJ}$ so that there exists a smooth family of K\"ahler metrics $\hat{\chi}_j=\hat{\chi}_j(t,\varrho)$ with constant coefficients such that
\begin{equation} \label{equivalence to forms with constant coefficients}
\hat{\chi}_j \leq \hat{\chi} \leq  \hat{\chi}_j+\hat\epsilon\, \xi
\end{equation}
on each $B_{j,4R}$ for all $t \in (0,1]$ and $\varrho \in (0,\varrho_0]$. We may further assume that for any fixed $R>0$ and a finite covering $\{B_{j,3R}\}_{j \in \cJ}$, there exists $r_0 \in (0,R)$ such that for all $r<r_0$,
\begin{equation} \label{bound for smoothing}
\chi-\frac{1}{10} \xi \leq \chi^{(r)} \leq \chi+\frac{1}{10} \xi, \quad \frac{9}{10} \xi \leq \xi^{(r)} \leq \frac{11}{10} \xi
\end{equation}
on each $B_{j,3R}$ since $\chi$ and $\xi$ are smooth and convergence of the smoothing is uniform on $B_{j,3R}$. Also we may assume that
\begin{equation} \label{bound for hat chi and xi}
\hat{\chi} \leq 2 \xi
\end{equation}
for all $t \in (0,1]$ and $\varrho \in (0,\varrho_0]$ since $\varrho_0<1/100$. In this section, we will show the following:

\begin{thm} \label{local subsolution}
Let $\d_0>0$ be a constant determined in Lemma \ref{mass concentration}. Then there exist constants $\hat{\epsilon}=\hat{\epsilon}(\omega_0,\chi,\d_0,\theta_0) \in (0,1]$, $\epsilon_0=\epsilon_0(\omega_0,\chi,\d_0,\theta_0)>0$ such that for any finite covering $\{B_{j,3R}\}$ and $r_0 \in (0,R)$ satisfying \eqref{equivalence to Euclidean metrics}, \eqref{equivalence to forms with constant coefficients} and \eqref{bound for smoothing}, there exists a closed $(1,1)$-current $\Omega \in \a-\d_0 \b$ satisfying the following properties:
\begin{enumerate}\setlength{\itemsep}{1mm}
\item $\Omega+A\chi$ is a K\"ahler current for some constant $A>0$, and has positive Lelong number along $E_0$.
\item For any $r \in (0,r_0)$, we have
\begin{equation} \label{local smoothing Q}
Q_\chi(\Omega^{(r)}) \leq \theta_0-\epsilon_0
\end{equation}
on each $B_{j,3R} \backslash E_0$.
\end{enumerate}
\end{thm}
In order to prove Theorem \ref{local subsolution}, we perform the truncation as in \cite{Che21,Son20}. For any $z \in \cY$ and $\eta>0$, we set
\[
\Omega_{s,\eta}(z):=\cV^{-1} \bigg( \int_{\pi_1^{-1}(z) \backslash \Delta_\eta} \tilde{\omega}_{\cY,s} \wedge \Im(\omega_V+\i \chi_V)^m+\int_{\pi_1^{-1}(z) \cap \Delta_\eta} K \chi_\cY \wedge \Im(\omega_V+\i \chi_V)^m \bigg).
\]
\begin{lem}\label{Lma-cot-theta}
For any $t \in (0,1]$, $\varrho \in (0,\varrho_0]$, $s \in (0,1]$ and $\eta>0$, we have
\begin{eqnarray*}
&&\frac{1}{\cot(Q_{\hat{\chi}}(\Omega_{s,\eta}(z)))-\cot (\tilde{\theta}_0)}\\
&& \leq \frac{1}{\cot(\theta_0)-\cot(\tilde{\theta}_0)}+ \frac{\cot(\theta_0)-\cot(\tilde{\theta}_0)}{1+\cot^2(\tilde{\theta}_0)} \cV^{-1} \int_{\pi_1^{-1}(z) \cap \Delta_\eta} \Im(\omega_V+\i \chi_V)^m.
\end{eqnarray*}
\end{lem}
\begin{proof}
By using the monotonicity and concavity of $\cot(Q_{\hat{\chi}}(\cdot))$ with \eqref{estimates Q product}, we compute
\begin{eqnarray*}
&&\frac{1}{\cot(Q_{\hat{\chi}}(\Omega_{s,\eta}(z)))-\cot (\tilde{\theta}_0)}\\
&& \leq \cV^{-1} \int_{\pi_1^{-1}(z) \backslash \Delta_\eta} \frac{\Im(\omega_V+\i \chi_V)^m}{\cot (Q_{\chi_H}(\tilde{\omega}_{\cY,s}))-\cot(\tilde{\theta}_0)}+\cV^{-1} \int_{\pi_1^{-1}(z) \cap \Delta_\eta} \frac{\Im(\omega_V+\i \chi_V)^m}{\cot (Q_{\chi_H}(K\chi_H))-\cot(\tilde{\theta}_0)}\\
&& \leq \cV^{-1} \int_{\pi_1^{-1}(z) \backslash \Delta_\eta} \frac{\Im(\omega_V+\i \chi_V)^m}{\cot (\tilde{\theta}_0-Q_{\chi_V}(\omega_V))-\cot(\tilde{\theta}_0)}+\cV^{-1} \int_{\pi_1^{-1}(z) \cap \Delta_\eta} \frac{\Im(\omega_V+\i \chi_V)^m}{\cot (Q_{\chi_H}(K\chi_H))-\cot(\tilde{\theta}_0)}\\
&&= \cV^{-1} \int_{\pi_1^{-1}(z) \backslash \Delta_\eta} \frac{\cot(Q_{\chi_V}(\omega_V))-\cot(\tilde{\theta}_0)}{1+\cot^2(\tilde{\theta}_0)} \Im(\omega_V+\i \chi_V)^m\\
&&+\cV^{-1} \int_{\pi_1^{-1}(z) \cap \Delta_\eta} \frac{\cot(\theta_0)-\cot(\tilde{\theta}_0)}{1+\cot^2(\tilde{\theta}_0)} \Im(\omega_V+\i \chi_V)^m.
\end{eqnarray*}
We apply the computation of Lemma \ref{estimates for omega s} (from third line to last line) to the first term, and obtain
\begin{eqnarray*}
&&\frac{1}{\cot(Q_{\hat{\chi}}(\Omega_{s,\eta}(z)))-\cot (\tilde{\theta}_0)}\\
&& \leq \frac{1}{\cot(\theta_0)-\cot(\tilde{\theta}_0)}-\cV^{-1} \int_{\pi_1^{-1}(z) \cap \Delta_\eta} \frac{\cot(Q_{\chi_V}(\omega_V))-\cot(\tilde{\theta}_0)}{1+\cot^2(\tilde{\theta}_0)} \Im(\omega_V+\i \chi_V)^m \\
&&+\cV^{-1} \int_{\pi_1^{-1}(z) \cap \Delta_\eta} \frac{\cot(\theta_0)-\cot(\tilde{\theta}_0)}{1+\cot^2(\tilde{\theta}_0)} \Im(\omega_V+\i \chi_V)^m\\
&&= \frac{1}{\cot(\theta_0)-\cot(\tilde{\theta}_0)}+\cV^{-1} \int_{\pi_1^{-1}(z) \cap \Delta_\eta} \frac{\cot(\theta_0)-\cot(Q_{\chi_V}(\omega_V))}{1+\cot^2(\tilde{\theta}_0)} \Im(\omega_V+\i \chi_V)^m\\
&& \leq \frac{1}{\cot(\theta_0)-\cot(\tilde{\theta}_0)}+ \frac{\cot(\theta_0)-\cot(\tilde{\theta}_0)}{1+\cot^2(\tilde{\theta}_0)} \cV^{-1} \int_{\pi_1^{-1}(z) \cap \Delta_\eta} \Im(\omega_V+\i \chi_V)^m,
\end{eqnarray*}
where we used $Q_{\chi_V}(\omega_V) \leq \tilde{\theta}_0$ (cf. \eqref{estimates Q product}) in the last inequality.
\end{proof}
Since
\[
\int_{\pi_1^{-1}(z) \cap \Delta_\eta} \Im(\omega_V+\i \chi_V)^m \leq \int_{\pi_1^{-1}(z)} \Im(\omega_V+\i \chi_V)^m=\cV,
\]
from the above lemma, we immediately obtain the following:
\begin{cor}
There exist constants $\tilde{\theta}_1 \in (\theta_0, \tilde{\theta}_0)$, (depending only on $\theta_0$) such that for any $t \in (0,1]$, $\varrho \in (0,\varrho_0]$, $s \in (0,1]$ and $\eta>0$, we have
\[
Q_{\hat{\chi}}(\Omega_{s,\eta}) \leq \tilde{\theta}_1.
\]
\end{cor}

\begin{lem} \label{Q estimate for smoothing}
There is $\sigma_{0}(\theta_0,n)\in(0,\delta_{0})$ such that we can choose $\hat{\epsilon}=\hat{\epsilon}(\sigma_{0},\theta_0) \in (0,1]$ so that the following property holds: assume that $\hat{\chi}$ satisfies \eqref{equivalence to forms with constant coefficients} on each $B_{j,4R}$ for the constant $\hat{\epsilon}$. Then for any $t \in (0,1]$, $\varrho \in (0,\varrho_0]$ and $r \in (0,R)$, there exist $\eta_0=\eta_0(\e,t,\varrho,r)>0$, $s_0=s_0(\e,t,\varrho,r) \in (0,1]$ such that for any $\eta \in (0,\eta_0)$, $s \in (0,s_0)$, we have
\[
Q_{\hat{\chi}} (\Omega_{s,\eta}^{(r)}+\sigma_{0} \xi ) < \theta_0
\]
on each $B_{j,3R}$.
\end{lem}
\begin{proof}
We first let $\sigma_0$ be the constant obtained from Proposition~\ref{uniform continuity}. Decreasing $\sigma_{0}$ if needed, we may assume $\sigma_{0}<\delta_{0}$. By Proposition~\ref{uniform continuity}, choosing $\hat{\epsilon}$ sufficiently small, for any $z \in B_{j,3R}$, $r \in (0,R)$ and $y \in B_r(0)$, we have
\[
Q_{\hat{\chi}_j}(\Omega_{s,\eta}+\sigma_0\xi)|_{(z+y)} < \tilde{\theta}_0.
\]
Since $\hat{\chi}_j$ has constant coefficients, combining the above with the concavity of $\cot(Q_{\hat{\chi}_j}(\cdot))$, Lemma~\ref{Lma-cot-theta} and Proposition~\ref{uniform continuity}, decreasing $\hat{\epsilon}$ if necessary, for any $z \in B_{j,3R}$ and $r \in (0,R)$, we have
\begin{equation}\label{estimate-lma-local}
\begin{split}
&\quad \frac{1}{\cot(Q_{\hat{\chi}_j}(\Omega_{s,\eta}^{(r)}(z)+\sigma_0 \xi^{(r)}(z)))-\cot(\tilde{\theta}_0)} \\
&\leq \int_{y \in B_r(0)} r^{-2m} \r(r^{-1}|y|) \frac{1}{\cot(Q_{\hat{\chi}_j}(\Omega_{s,\eta}+\sigma_0\xi)|_{(z+y)})-\cot(\tilde{\theta}_0)} d \Vol_{g_j}(y)\\
&\leq \int_{y \in B_r(0)} r^{-2m} \r(r^{-1}|y|) \frac{1}{\cot(Q_{\hat{\chi}}(\Omega_{s,\eta})(z+y))-\cot(\tilde{\theta}_0)} d \Vol_{g_j}(y)\\
&\leq    \frac{1}{\cot(\theta_0)-\cot(\tilde{\theta}_0)}+ \frac{\cot(\theta_0)-\cot(\tilde{\theta}_0)}{1+\cot^2(\tilde{\theta}_0)}2^m r^{-2m} \cV^{-1} \int_{\Delta_\eta} \Im(\omega_V+\i \chi_V)^m \wedge \pi_1^\ast \xi^m.
\end{split}
\end{equation}
We note that
\[
\int_{\Delta_\eta} \Im(\omega_V+\i \chi_V)^m \wedge \pi_1^\ast \xi^m=\int_{\Delta_\eta} \Im(\omega_{\cY,s}+\i \chi_V)^m \wedge \pi_1^\ast \xi^m.
\]
So this is a linear combination of
\[
\int_{\Delta_\eta} \omega_{\cY,m-2k-1,s}^{\pm} \wedge \chi_V^{2k+1} \wedge \pi_1^\ast \xi^m, \quad k=0,\ldots, \bigg\lfloor \frac{m-1}{2} \bigg\rfloor,
\]
which has no concentration of mass on $\Delta$ by Lemma \ref{mass concentration}. This yields that for any $\varepsilon>0$, $t \in (0,1]$, $\varrho \in (0,\varrho_0]$ and $r \in (0,R)$, there exist $\eta_0=\eta_0(\varepsilon,t,\varrho,r)>0$, $s_0=s_0(\varepsilon,t,\varrho,r) \in (0,1]$ such that for any $\eta \in (0,\eta_0)$, $s \in (0,s_0)$ we have
\[
\frac{\cot(\theta_0)-\cot(\tilde{\theta}_0)}{1+\cot^2(\tilde{\theta}_0)}2^m r^{-2m} \cV^{-1} \int_{\Delta_\eta} \Im(\omega_V+\i \chi_V)^m \wedge \pi_1^\ast \xi^m \leq \frac{\varepsilon}{2}.
\]
Otherwise, there exists some $k$, a constant $\e_3>0$ and sequences $\eta_i \to 0$, $s_i \to 0$ (we may further assume that these sequences are monotone) such that
\[
\pm \int_{\Delta_{\eta_i}} \omega_{\cY,m-2k-1,s_i}^{\pm} \wedge \chi_V^{2k+1} \wedge \pi_1^\ast \xi^m \geq \e_3
\]
holds for all $i$. In particular, for any $i \geq i'$ we have
\[
\pm \int_{\Delta_{\eta_{i'}}} \omega_{\cY,m-2k-1,s_i}^{\pm} \wedge \chi_V^{2k+1} \wedge \pi_1^\ast \xi^m \geq \e_3.
\]
This yields a contradiction by letting $i \to \infty$ and followed by $i' \to \infty$.

By \eqref{estimate-lma-local}, we may choose $\varepsilon$ sufficiently small, so that $Q_{\hat{\chi}_j}(\Omega_{s,\eta}^{(r)}+\sigma_0 \xi^{(r)})<\theta_0+\frac1{100} c_0\sigma_0$ where $c_0$ is the constant from Proposition~\ref{semi-continuity}. By Proposition~\ref{monotone and concave}, \eqref{bound for hat chi and xi} and \eqref{bound for smoothing},  we conclude that
\[
Q_{\hat{\chi}_j}(\Omega_{s,\eta}^{(r)}+10\sigma_0 \xi)
\leq Q_{\hat{\chi}_j}(\Omega_{s,\eta}^{(r)}+\sigma_0 \xi^{(r)})-5c_0\sigma_0<\theta_0.
\]
Using Proposition \ref{uniform continuity} and \eqref{equivalence to forms with constant coefficients} again, we obtain
\[
Q_{\hat{\chi}}(\Omega_{s,\eta}^{(r)}+100\sigma_0 \xi)<\theta_0.
\]
This completed the proof by relabelling the constants.
\end{proof}

As with \eqref{fiber integral}, we observe that
\[
\sum_{k=0}^{\lfloor \frac{m-1}{2} \rfloor}(-1)^k \frac{m!}{(m-2k)!(2k+1)!} \int_{\pi_1^{-1}(z) \cap \Delta_\eta} \omega_{\cY,s}^{m-2k} \wedge \pi_2^\ast \hat{\chi}^{2k+1}=\int_{\pi_1^{-1}(z) \cap \Delta_\eta} \tilde{\omega}_{\cY,s} \wedge \Im(\omega_V+\i \chi_V)^m.
\]
Now we rewrite $\Omega_{s,\eta}$ as
\[
\Omega_{s,\eta}=\omega_s-\omega_{s,\eta}-\omega_{s,\eta}'+\omega_{s,\eta}'',
\]
where
\[
\omega_{s,\eta}(z):=\cV^{-1} \int_{\pi_1^{-1}(z) \cap \Delta_\eta} \omega_{\cY,m,s}^+ \wedge \pi_2^\ast \hat{\chi},
\]
\begin{eqnarray*}
\omega_{s,\eta}'(z)&:=& \cV^{-1} \int_{\pi_1^{-1}(z) \cap \Delta_\eta} \omega_{\cY,m,s}^- \wedge \pi_2^\ast \hat{\chi}\\
&+& \cV^{-1} \sum_{k=1}^{\lfloor \frac{m-1}{2} \rfloor} (-1)^k \frac{m!}{(m-2k)!(2k+1)!} \int_{\pi_1^{-1}(z) \cap \Delta_\eta} \omega_{\cY,s}^{m-2k} \wedge \pi_2^\ast \hat{\chi}^{2k+1},
\end{eqnarray*}
\[
\omega_{s,\eta}''(z):=\cV^{-1} \int_{\pi_1^{-1}(z) \cap \Delta_\eta} K \chi_\cY \wedge \Im(\omega_V+\i \chi_V)^m.
\]
and we will consider each term separately.
\begin{lem} \label{mass concentration for omega s eta}
For any $t \in (0,1]$, $\varrho \in (0,\varrho_0]$, $r \in (0,r_0)$, $\eta>0$ and $z \in \hat{Y}$, there exist $s_1=s_1(t,\varrho,r,\eta,z)>0$ such that for all $s \in (0,s_1)$ we have
\[
(\omega_{s,\eta}+100 \d_0 \dd \phi)^{(r)} \geq 80 \d_0 \xi
\]
at $z \in B_{j,3R}$.
\end{lem}
\begin{proof}
Assume the statement does not hold at some point $z \in B_{j,3R}$. So there exists a sequence $s_i \to 0$ and $v_i \in T_{z}^{1,0}\hat{Y}$ with $|v_i|_{g_j}=1$ such that for all $i$ we have
\[
(\omega_{s_i,\eta}+100 \d_0 \dd \phi)^{(r)}(v_i, \overline{v_i})<80 \d_0 \xi (v_i, \overline{v_i})
\]
at $z \in B_{j,3R}$. After passing to a further subsequence, we may assume that $\Psi_{\cY,s}^m$ converges weakly to a closed positive current $\Xi_m \geq 100 \cV \d_0[\Delta]$. Then we compute
\begin{eqnarray*}
&& \lim_{i \to \infty} \omega_{s_i,\eta} {}^{(r)}( z)\\
&& = \lim_{i \to \infty} \int_{y \in B_r(0)} r^{-2m} \r(r^{-1} |y|) \omega_{s_i,\eta} (z+y) d \Vol_{g_j}(y)\\
&& = \cV^{-1} \lim_{i \to \infty} \int_{y \in B_r(0)} r^{-2m} \r(r^{-1} |y|) \bigg( \int_{w \in \pi_1^{-1}(z+y) \cap \Delta_\eta} \omega_{\cY,m,s_i}^+(z+y,w) \wedge \chi_V(w) \bigg) d \Vol_{g_j}(y)\\
&& \geq \cV^{-1} \lim_{i \to \infty} \int_{y \in B_r(0)} r^{-2m} \r(r^{-1} |y|) \bigg( \int_{w \in \pi_1^{-1}(z+y) \cap \Delta_\eta} \Psi_{\cY,s_i}^m (z+y,w) \wedge \chi_V(w) \bigg) d \Vol_{g_j}(y)\\
&& = \cV^{-1} \lim_{i \to \infty} \int_{(y,w) \in (B_r(0) \times \hat{Y}) \cap \Delta_\eta} r^{-2m} \r(r^{-1} |y|) \Psi_{\cY,s_i}^m(z+y,w) \wedge \chi_V(w) d \Vol_{g_j}(y)\\
&& \geq 100 \d_0 \int_{(y,w) \in (B_r(0) \times \hat{Y}) \cap \Delta} r^{-2m} \r(r^{-1} |y|) \chi_V(w) d \Vol_{g_j}(y)\\
&& =  100 \d_0 \int_{y \in B_r(0)} r^{-2m} \r(r^{-1} |y|) \hat{\chi}(z+y) d \Vol_{g_j}(y),
\end{eqnarray*}
where we used $w=z+y$ on $B_{j,4R}$ in the last equality. So by definition of $\phi$, we have
\begin{eqnarray*}
&& \lim_{i \to \infty} (\omega_{s_i,\eta}+100 \d_0 \dd \phi)^{(r)}(z)\\
&& \geq 100 \d_0 \int_{y \in B_r(0)} r^{-2m} \r(r^{-1}|y|)(\hat{\chi}+\dd \phi)(z+y) d \Vol_{g_j}(y)\\
&& \geq 100 \d_0 \xi^{(r)}(z)\\[3mm]
&& \geq 90 \d_0 \xi (z)
\end{eqnarray*}
by \eqref{bound for smoothing}. In particular, for sufficiently large $i$, we have
\[
(\omega_{s_i,\eta}+100 \d_0 \dd \phi)^{(r)}(z)(v_i, \overline{v_i}) \geq 85 \d_0 \xi (z) (v_i, \overline{v_i}).
\]
This is a contradiction.
\end{proof}
\begin{lem} \label{no mass concentration 1}
For any $t \in (0,1]$, $\varrho \in (0,\varrho_0]$, $r \in (0,R)$ and $z \in \hat{Y}$, there exists $s_2=s_2(t,\varrho,r,z)>0$, $\eta_1=\eta_1(t,\varrho,r,z)>0$ such that for any $s \in (0,s_2)$, $\eta \in (0,\eta_1)$, we have
\[
\omega_{s,\eta}'' {}^{(r)} \leq \d_0 \xi
\]
at $z \in B_{j,3R}$.
\end{lem}
\begin{proof}
Suppose that the statement does not hold. So there exist sequences $s_i \to 0$, $\eta_i \to 0$ and $v_i \in T_z^{1,0}\hat{Y}$ with $|v_i|_{g_j}=1$ such that
\[
\omega_{s_i,\eta_i}'' {}^{(r)}(v_i, \overline{v_i})>\d_0 \xi (v_i, \overline{v_i})
\]
at some point $z \in B_{j,3R}$. Let $\g$ be any positive $(m-1,m-1)$-form on $B_{j,4R}$ with constant coefficients. We regard $\omega_{s_i,\eta_i}'' {}^{(r)} \wedge \g$ as a real number by using the Euclidean volume form $\Vol_{g_j}$ at $z$. Then after passing to a further subsequence, we have
\begin{eqnarray*}
&& \lim_{i \to \infty} (\omega_{s_i,\eta_i}'' {}^{(r)} \wedge \g) (z)\\
&&=\lim_{i \to \infty} \int_{y \in B_r(0)} r^{-2m} \r(r^{-1}|y|) \omega_{s_i,\eta_i}'' (z+y) \wedge \g(z+y)\\
&&=\cV^{-1} \lim_{i \to \infty} \int_{y \in B_r(0)} r^{-2m} \r(r^{-1}|y|) \bigg( \int_{w \in \pi_1^{-1}(z_i+y) \cap \Delta_{\eta_i}}\big( K \chi_H \wedge \Im(\omega_{\cY,s_i}+\i \chi_V)^m \big) (z+y,w) \bigg)\\
&& \wedge \g (z+y)\\
&& \leq \cV^{-1} r^{-2m} \lim_{i \to \infty} \int_{\Delta_{\eta_i}} K \chi_H \wedge \Im(\omega_{\cY,s_i}+\i \chi_V)^m \wedge \r(r^{-1}|y|) \pi_1^\ast \g \\
&&=0.
\end{eqnarray*}
since the last integral has no concentration of mass on $\Delta$ by the similar observation as in the proof of Lemma \ref{Q estimate for smoothing}. Thus for sufficiently large $i$, we have
\[
\omega_{s_i,\eta_i}'' {}^{(r)}(z)(v_i, \overline{v_i}) \leq \frac{\d_0}{2} \xi(z_i) (v_i, \overline{v_i}).
\]
This is a contradiction.
\end{proof}
Recall that the form $\omega_{s,\eta}'(z)$ can be written as the linear combination of the fiber integral of
\[
\omega_{\cY,m-2k,s}^+ \wedge \pi_2^\ast \hat{\chi}^{2k+1}, \; k=1,\ldots, \bigg\lfloor \frac{n-1}{2} \bigg\rfloor
\]
and
\[
\omega_{\cY,m-2k,s}^- \wedge \pi_2^\ast \hat{\chi}^{2k+1}, \; k=0,\ldots, \bigg\lfloor \frac{n-1}{2} \bigg\rfloor
\]
on $\pi_1^{-1}(z) \cap \Delta_\eta$, which have no concentration of mass on $\Delta$. Thus we can show the following lemma in the same way as Lemma \ref{no mass concentration 1}:
\begin{lem} \label{no mass concentration 2}
For any $t \in (0,1]$, $\varrho \in (0,\varrho_0]$, $r \in (0,R)$ and $z \in \hat{Y}$, there exists $s_3=s_3(t,\varrho,r,z)>0$, $\eta_2=\eta_2(t,\varrho,r,z)>0$ such that for any $s \in (0,s_3)$, $\eta \in (0,\eta_2)$, we have
\[
\omega_{s,\eta}' {}^{(r)} \geq -\d_0 \xi
\]
at $z \in B_{j,3R}$.
\end{lem}

\begin{lem} \label{theta zero plus 2 epsilon}
There exists a constant $\hat{\epsilon}=\hat{\epsilon}(\sigma_{0},\theta_0) \in (0,1]$ satisfying the following properties: assume that for $\hat{\epsilon}$, the form $\hat{\chi}$ satisfies \eqref{equivalence to forms with constant coefficients} on each $B_{j,4R}$. Then for any $t \in (0,1]$, $\varrho \in (0,\varrho_0]$, $r \in (0,r_0)$ and $z \in \hat{Y}$, there exists $s_4=s_4(\e,t,\varrho,r,z)>0$ such that for any $s \in (0,s_4)$, we have
\[
Q_{\hat{\chi}}\left( \big( \omega_s-25\d_0 \hat{\chi}+100 \d_0 \dd \phi \big)^{(r)} \right) < \theta_0
\]
at $z \in B_{j,3R}$.
\end{lem}

\begin{proof}
For any given $t \in (0,1]$, $\varrho \in (0,\varrho_0]$, $r \in (0,r_0)$ and $z \in \hat{Y}$, constants $s_0$, $s_2$, $s_3$, $\eta_0$, $\eta_1$, $\eta_2$ are determined by the previous lemmas. We fix $\eta>0$ so that $\eta<\min \{\eta_0,\eta_1,\eta_2 \}$. Then for this $\eta$, we have $s_1=s_1(t,\varrho,r,\eta,z)$ as in Lemma \ref{mass concentration for omega s eta}. We take $s_4>0$ so that $s_4<\min\{s_0,s_1,s_2,s_3 \}$. Then for any $s \in (0,s_4)$, we have
\begin{eqnarray*}
&& \big( \omega_s-25\d_0 \hat{\chi}+100 \d_0 \dd \phi \big)^{(r)}\\
&&=\big( \Omega_{s,\eta}+(\omega_{s,\eta}+100 \d_0 \dd \phi)-25 \d_0 \hat{\chi}+\omega_{s,\eta}'-\omega_{s,\eta}'' \big)^{(r)}\\[1mm]
&& \geq \Omega_{s,\eta}^{(r)}+80\d_0 \xi-55 \d_0 \xi-\d_0 \xi-\d_0 \xi\\[1mm]
&& \geq \Omega_{s,\eta}^{(r)}+\sigma_{0}\xi
\end{eqnarray*}
at $z \in B_{j,3R}$, where we used \eqref{bound for smoothing}, \eqref{bound for hat chi and xi} in the third line, and $\sigma_{0}<\delta_{0}$ in the last line. Then Lemma \ref{Q estimate for smoothing} shows
\[
Q_{\hat{\chi}} \left(\big( \omega_s-25\d_0 \hat{\chi}+100 \d_0 \dd \phi \big)^{(r)}\right)
\leq Q_{\hat{\chi}} \big(\Omega_{s,\eta}^{(r)}+\sigma_{0}\xi\big) <\theta_{0}
\]
at $z \in B_{j,3R}$ as desired.
\end{proof}

Now we are in a position to prove Theorem \ref{local subsolution}.

\begin{proof}[Proof of Theorem \ref{local subsolution}]
By Lemma \ref{estimates for omega s}, we have $\omega_s-\cot(\theta_0) \hat{\chi} \geq 0$ and so
\begin{eqnarray*}
\omega_s-25\d_0 \hat{\chi}+100\d_0 \dd \phi +A \hat{\chi} \geq 100 \d_0 \xi
\end{eqnarray*}
for $A:=|\cot(\theta_0)|+200\d_0>0$. Thus after passing to a subsequence, we have a closed current
\[
\omega^\star=\omega^\star(t,\varrho):=\lim_{s_i \to 0}(\omega_{s_i}-25\d_0 \hat{\chi}+100\d_0 \dd \phi ) \in \hat{\a}-25\d_0 \hat{\b}.
\]
On each $B_{j,4R}$, we take a quasi-PSH function $h_{s_i}$ such that
\[
\omega_{s_i}-25\d_0 \hat{\chi}+100\d_0 \dd \phi=\dd h_{s_i}.
\]
Then after passing to a subsequence, we may further assume that $h_{s_i}$ converges to some quasi-PSH function $h^\star$ in $L^1$ with $\omega^\star=\dd h^\star$. Thus for any $r \in (0,r_0)$, $h_{s_i}^{(r)}$ converges to $h^\star {}^{(r)}$ uniformly on $B_{j,3R}$. On the other hand, Lemma \ref{theta zero plus 2 epsilon} implies that for any $t \in (0,1]$, $\varrho \in (0,\varrho_0]$, $r \in (0,r_0)$ and $z \in \hat{Y}$, there exists $s_4>0$ such that
\[
Q_{\hat{\chi}}(\dd h_{s_i} {}^{(r)}) < \theta_0
\]
hold at $z \in B_{j,3R}$ for all $i$ satisfying $s_i<s_4$. Thus for any $t \in (0,1]$, $\varrho \in (0,\varrho_0]$ and $r \in (0,r_0)$, we have
\[
Q_{\hat{\chi}}(\dd h^\star {}^{(r)}) \leq \theta_0
\]
on $B_{j,3R}$. The current $\omega^\star$ inherits a lower bound
\[
\omega^\star+A \hat{\chi} \geq 100 \d_0 \xi.
\]
So we can take a subsequence $t_k, \varrho_k \to 0$ so that $\omega^\star (t_k,\varrho_k)$ converges to a closed current $\omega_0^\star \in \a-25\d_0 \b$ with
\begin{equation} \label{lower bound of star omega}
\omega_0^\star+A \chi \geq 100 \d_0 \xi.
\end{equation}
By the similar argument as above, we know that the current $\omega_0^\star \in \a-25\d_0 \b$ satisfies
\[
Q_\chi(\omega_0^\star {}^{(r)}) \leq \theta_0
\]
for all $r \in (0,r_0)$ and $B_{j,3R} \backslash E_0$. Now we define a current $\Omega \in \a-\d_0 \b$ as
\[
\Omega:= \omega_0^\star+24\d_0 \chi+24\d_0 \dd\phi.
\]
This implies that
\begin{equation} \label{lower bound of Omega}
\Omega \geq \omega_0^\star+24\d_0 \xi.
\end{equation}
Combining with \eqref{bound for smoothing}, we have
\[
\Omega^{(r)} \geq \omega_0^\star {}^{(r)}+24\d_0 \xi^{(r)} \geq \omega_0^\star {}^{(r)}+\d_0 \chi.
\]
By Proposition \ref{semi-continuity}, we set $\epsilon_0=c_0\delta_0$ and obtain
\[
Q_\chi(\Omega^{(r)}) \leq Q_\chi \big( \omega_0^\star {}^{(r)}+\d_0 \chi \big)
\leq \theta_0-c_0\delta_0 = \theta_0-\epsilon_0
\]
on $B_{j,3R} \backslash E_0$. On the other hand, by \eqref{lower bound of star omega} and \eqref{lower bound of Omega}, we observe that
\[
\Omega+A \chi \geq 100 \d_0 \xi.
\]
We note that $\Omega+A\chi$ has positive Lelong number along $E_0$ since $\phi$ is so. This completes the proof.
\end{proof}

\section{Gluing construction} \label{gluing construction}
In this section, we perform a gluing construction to the local subsolutions obtained in Theorem \ref{local subsolution}, and show the following:
\begin{thm} \label{gluing of local subsolutions}
There exists $\varphi_Y \in C^\infty (Y \backslash Y_{\sing})$ such that
\begin{enumerate}\setlength{\itemsep}{1mm}
\item The $(1,1)$-form $\omega_Y:=\omega_0+\dd \varphi_Y$ on $Y \backslash Y_{\sing}$ satisfies
\[
Q_\chi(\omega_Y)<\theta_0.
\]
\item The function $\varphi_Y$ tends to $-\infty$ uniformly at $Y_{\sing}$.
\end{enumerate}
\end{thm}
Let $\Omega$ be a closed current constructed in Theorem \ref{local subsolution}. Continuing from the previous section, we have to deal with the problem that $\Omega$ is no longer K\"ahler. So we take a constant $A>0$ such that $\Omega+A\chi$ is a K\"ahler current and $\omega_0-\d_0 \chi+A\chi$ is semipositive. We take a potential $\varphi$ so that
\[
\Omega=\omega_0-\d_0 \chi+\dd \varphi,
\]
and hence $\varphi \in \PSH(\hat{Y}, \omega_0-\d_0 \chi+A\chi)$. For a constant $\tau \in (0,\frac{1}{10000})$, by taking sufficiently small $R>0$ and finer covering $\{B_{j,3R}\}$ if necessary, we may assume that
\[
(1-\tau)g_j \leq \xi \leq (1+\tau) g_j,
\]
\[
\xi=\dd \phi_{j, \xi}, \quad \big|\phi_{j,\xi}-|z|^2 \big| \leq \tau R^2, \quad \phi_{j,\xi}(0)=0
\]
on $B_{j,4R}$, where $\phi_{j, \xi}$ is a local potential for $\xi$, and $|z|^2$ denotes the square of the Euclidean distance to $z$ from the origin. In other words, each $B_{j,4R}$ is very close to the normal coordinates with respect to $\xi$. On each $B_{j,4R}$, we choose a local bounded potentials of $\omega_0$ and $\chi$ so that
\[
\omega_0=\dd \phi_{j,\omega_0}, \quad \phi_{j,\omega_0}(0)=0,
\]
\[
\chi=\dd \phi_{j,\chi}, \quad \phi_{j,\chi}(0)=0,
\]
and assume that
\[
|\nabla (\phi_{j,\omega_0}-\d_0 \phi_{j,\chi}+A\phi_{j,\chi})| \leq KR
\]
for some fixed large constant $K>0$. We choose a constant $\mu_0$ so that
\begin{equation} \label{choice of mu zero}
\mu_0<\frac{\d_0^3 R^2}{8A}.
\end{equation}
By replacing $r_0$ with a smaller number if necessary, we may assume that
\begin{equation} \label{two sided bound of mu zero}
\phi_{j,\chi}-\mu_0 \leq \phi_{j,\chi}^{(r)} \leq \phi_{j,\chi}+\mu_0
\end{equation}
on each $B_{j,3R}$ for all $r \in (0,r_0)$. We set
\[
\Omega=\dd \varphi_j, \quad \varphi_j:=\phi_{j,\omega_0}-\d_0 \phi_{j,\chi}+\varphi
\]
on $B_{j,4R}$ and
\[
\tilde{\varphi}_{j,r}:=\varphi_j^{(r)}-\phi_{j,\omega_0}+\d_0 \phi_{j,\chi}-\d_0^3 \phi_{j,\xi}+\d_0^2 \phi
\]
for $r \in (0,R)$ on $B_{j,3R}$, where the function $\phi$ is defined in \eqref{def-phi}. Since $\d_0 \in (0,\kappa_{0})$, we have
\[
\chi+\d_0 \dd \phi=(1-\d_0)\chi+\d_0 \xi \geq \d_0 \xi \geq \d_0^2 \xi.
\]
Thus
\[
\omega_0+\dd \tilde{\varphi}_{j,r}=\Omega^{(r)}-\d_0^3 \xi+\d_0(\chi+\d_0 \dd \phi) \geq \Omega^{(r)},
\]
which yields that
\begin{equation} \label{local subsolution away from e zero}
Q_\chi(\omega_0+\dd \tilde{\varphi}_{j,r}) \leq \theta_0-\epsilon_0
\end{equation}
on each $B_{i,3R} \backslash E_0$ for all $r \in (0,r_0)$. For any $z \in B_{j,3R}$ and $r \in (0,R)$, we set a PSH function $\psi_j$ and its approximation $\psi_{j,r}$ as
\[
\psi_j:=\varphi_j+A\phi_{j,\chi}, \quad \psi_{j,r}(z):=\sup_{B_{j,r}(z)} \psi_j,
\]
where $B_{j,r}(z)$ denotes a Euclidean ball of radius $r$ centered at $z$ in $B_{j,4R}$. For any $r \in (0,R/2)$, we define
\[
\nu_{j,\psi_j} (z,r):=\frac{\psi_{j,\frac{3}{4}R}(z)-\psi_{j,r}(z)}{\log \big( \frac{3}{4}R \big)-\log r}, \quad z \in B_{j, 3R}.
\]
As $r \to 0$, the quantity $\nu_{j,\psi_j}(z,r)$ converges decreasingly to the Lelong number of $\varphi$ at $z$ (with respect to a reference form $\omega_0-\d_0 \chi+A \chi$):
\[
\lim_{r \to 0} \nu_{j,\psi_j}(z,r)=\nu_\varphi(z),
\]
where we note that the Lelong number $\nu_\varphi(z)$ is independent of $j \in \cJ$. We will use the following comparison formula for these functions which was proved in \cite[Lemma 4.2]{Che21} based on the computation of \cite{BK07}:
\begin{lem} \label{comparison formulas}
For any $r \in (0,R/2)$ and $z \in B_{j,3R}$, the following estimates hold:
\begin{enumerate}
\item $0 \leq \psi_{j,r}(z)-\psi_{j,\frac{r}{2}}(z) \leq (\log 2) \nu_{j,\psi_j}(z,r)$,
\item $0 \leq \psi_{j,r}(z)-\psi_j^{(r)}(z) \leq \eta \nu_{j,\psi_j}(z,r)$,
where the function $\r$ is defined in Definition \ref{local smoothing}, and a constant $\eta>0$ is defined by
\[
\eta:=\frac{3^{2m-1}}{2^{2m-3}} \log2+\Vol(\p B_1(0)) \int_0^1 \log \bigg(\frac{1}{t} \bigg) t^{2m-1} \r(t) dt.
\]
\end{enumerate}
\end{lem}
We let $E_{\tilde{\epsilon}}$ be the analytic subvariety of $\hat{Y}$ defined by
\[
E_{\tilde{\epsilon}}:=\{z \in \hat{Y}|\nu_\varphi(z) \geq \tilde{\epsilon} \},
\]
where a constant $\tilde{\epsilon}>0$ is taken so that
\begin{equation} \label{choice of tilde epsilon}
\tilde{\epsilon}<\min \bigg\{ \frac{\d_0^3 R^2}{6+4\eta}, \frac{1}{4} \inf_{z \in E_0} \nu_\varphi(z) \bigg\}.
\end{equation}
We note that $E_{\tilde{\epsilon}}$ is an analytic subvariety of $\hat{Y}$ by Siu's result \cite{Siu74}, and contains $E_0$ by \eqref{choice of tilde epsilon}. Since $\dim \Phi(E_{\tilde{\epsilon}}) \leq \dim E_{\tilde{\epsilon}} \leq m-1$, we apply the induction assumption to $\Phi(E_{\tilde{\epsilon}})$. So there exist a real $(1,1)$-form $\omega_U \in \a|_{\Phi(E_{\tilde{\epsilon}}),X}$ such that
\[
\cQ_{\chi,n,n}(\omega_U)<\theta_0
\]
on some open neighbourhood $U$ of $\Phi(E_{\tilde{\epsilon}})$ in $X$, and hence
\[
\cQ_{\chi,m,n}(\omega_U) \leq \theta_0-\epsilon_0
\]
on $U$ by replacing $\epsilon_0>0$ and shrinking $U$ if necessary. This implies that
\[
\cQ_{\chi,m,n}(\Phi^\ast \omega_U) \leq \theta_0-\epsilon_0
\]
on $\Phi^{-1}(U) \backslash \Phi^{-1}(Y_{\sing})$. We set $\hat{U}:=\Phi^{-1}(U) \cap \hat{Y}$ and $\omega_{\hat{U}}:=(\Phi^\ast \omega_U)|_{\hat{U}}$. By \eqref{variational characterization for Q}, we observe that $\hat{U}$ is a open neighbourhood of $E_{\tilde{\epsilon}}$ in $\hat{Y}$ and satisfies
\begin{equation} \label{local subsolution near e zero}
Q_\chi(\omega_{\hat{U}}) \leq \theta_0-\epsilon_0
\end{equation}
on $\hat{U} \backslash E_0$. We write $\omega_{\hat{U}}$ as
\[
\omega_{\hat{U}}:=\omega_0+\dd \varphi_{\hat{U}}
\]
for some uniformly bounded, smooth function $\varphi_{\hat{U}}$ on $\hat{U}$.
\begin{lem} \label{gluing condition}
There exist $r_1 \in (0,\min\{r_0,R/2\})$ and an open neighbourhood $\hat{V} \Subset \hat{U}$ of $E_{\tilde{\epsilon}}$ in $\hat{Y}$ such that the following estimates hold for all $r \in (0,r_1)$:
\begin{enumerate}
\item If $z \in \hat{Y} \backslash \hat{V}$, then
\begin{equation} \label{small Lelong number}
\max_{\{j \in \cJ|z \in B_{j,3R}\}} \nu_{j, \psi_j}(z,r) \leq 2 \tilde{\epsilon}
\end{equation}
and
\[
\max_{\{j \in \cJ|z \in B_{j,3R}\}} \tilde{\varphi}_{j,r}(z) \geq \sup_{\hat{U}} \varphi_{\hat{U}}+3\tilde{\epsilon} \log r+1.
\]
\item If
\[
\max_{\{j \in \cJ|z \in B_{j,3R}\}} \nu_{j, \psi_j}(z,r) \geq 4 \tilde{\epsilon},
\]
then
\[
\max_{\{j \in \cJ|z \in B_{j,3R}\}} \tilde{\varphi}_{j,r}(z) \leq \inf_{\hat{U}} \varphi_{\hat{U}}+3 \tilde{\epsilon} \log r-1.
\]
\item If
\[
\max_{\{j \in \cJ|z \in B_{j,3R}\}} \nu_{j, \psi_j}(z,r) \leq 4 \tilde{\epsilon},
\]
then
\[
\max_{\{j \in \cJ | z \in B_{j,3R} \backslash B_{j,2R}\}} \tilde{\varphi}_{j,r}(z) \leq \max_{\{j \in \cJ | z \in B_{j,R}\}} \tilde{\varphi}_{i,r}(z)-2\tilde{\epsilon}.
\]
\end{enumerate}
\end{lem}
\begin{proof}
{The main difference from \cite[Lemma 6.3]{Son20} is that we need the additional term $A \phi_{j,\chi}^{(r)}$ to deal with $\nu_{j,\psi_j}(z,r)$. However, one can prove in the same way as in \cite[Lemma 6.3]{Son20} by using Lemma \ref{comparison formulas} and the two sided bound \eqref{two sided bound of mu zero}. We give the full proof for the sake of completeness.

\noindent
(1) We fix any open $\hat{V} \Subset \hat{U}$. Then we can choose $r_1 \in (0,\min\{r_0,R/2\})$ sufficiently small so that \eqref{small Lelong number} holds since the Lelong number of $\varphi$ at any point in $\hat{Y} \backslash \hat{V}$ is less than $\tilde{\epsilon}$, and $\nu_{j,\psi_j}(\cdot,r)$ is increasing in $r>0$ and upper semi-continuous for any fixed $r>0$. So for any $z \in \hat{Y} \backslash \hat{V}$ and $j \in \cJ$ such that $z \in B_{j,3R}$, we have $\nu_{j,\psi_j}(z,r) \leq 2 \tilde{\epsilon}$. By the definition of $\nu_{j,\psi_j}(z,r)$, for all $r \in (0,r_1)$, we observe that
\[
\psi_{j,r}(z) \geq 2 \tilde{\epsilon} \log r-C_1
\]
for some constant $C_1>0$ depending only on $R$ and $\psi_{j,\frac{3}{4} R}$. This shows that
\begin{eqnarray*}
&& \tilde{\varphi}_{j,r}(z)-\sup_{\hat{U}} \varphi_{\hat{U}}-3 \tilde{\epsilon} \log r-1\\
&&=\psi_j^{(r)}(z)-A \phi_{j,\chi}^{(r)}(z)-\phi_{j,\omega_0}(z)+\d_0 \phi_{j,\chi}(z)-\d_0^3 \phi_{j,\xi}(z)+\d_0^2 \phi(z)-\sup_{\hat{U}} \varphi_{\hat{U}}-3 \tilde{\epsilon} \log r-1\\
&& \geq -\tilde{\epsilon} \log r-\eta \nu_{j,\psi_j}(z,r)-A \phi_{j,\chi}^{(r)}(z)-\phi_{j,\omega_0}(z)+\d_0 \phi_{j,\chi}(z)-\d_0^3 \phi_{j,\xi}(z)+\d_0^2 \phi(z)-\sup_{\hat{U}} \varphi_{\hat{U}}-1-C_1\\
&& \geq -\tilde{\epsilon} \log r-2 \tilde{\epsilon} \eta-A\sup_{B_{j,3R}}\phi_{j,\chi}-A\mu_0-\sup_{B_{j,3R}}\phi_{j,\omega_0}+\d_0 \sup_{B_{j,3R}}\phi_{j,\chi}-\d_0^3 \sup_{B_{j,3R}} \phi_{j,\xi}+\d_0^2 \inf_{\hat{Y} \backslash \hat{V}} \phi\\
&& -\sup_{\hat{U}} \varphi_{\hat{U}}-1-C_1\\
&&\geq -\tilde{\epsilon} \log r-C_2,
\end{eqnarray*}
where the constant $C_2>0$ depends only on $R$, $\phi_{j,\omega_0}$, $\phi_{j,\chi}$, $\phi_{j,\xi}$, $\phi|_{\hat{Y} \backslash \hat{V}}$ and $\varphi_{\hat{U}}$. Thus we may choose $r_1<e^{-C_2/\tilde{\epsilon}}$.

\noindent
(2) We will prove by contradiction. So assume that there exists $j \in \cJ$ such that $z \in B_{j,3R}$ and
\[
\tilde{\varphi}_{j,r}(z)>\inf_{\hat{U}} \varphi_{\hat{U}}+3 \tilde{\epsilon} \log r-1.
\]
Then for any $r \in (0, \min\{r_0,R/2 \})$, we compute
\begin{eqnarray*}
&& \psi_{j,r}(z) \geq \psi_j^{(r)}(z)\\[1mm]
&&=\tilde{\varphi}_{j,r}(z)+\phi_{j,\omega_0}(z)-\d_0 \phi_{j,\chi}(z)+\d_0^3 \phi_{j,\xi}(z)-\d_0^2 \phi(z)+A \phi_{j,\chi}^{(r)}(z)\\[2.5mm]
&& \geq 3 \tilde{\epsilon} \log r+\inf_{B_{j,3R}}\phi_{j,\omega_0}-\d_0 \sup_{B_{j,3R}} \phi_{j,\chi}+\d_0^3 \inf_{B_{j,3R}} \phi_{j,\xi}-\d_0^2 \sup_{B_{j,3R}} \phi+A \inf_{B_{j,3R}} \phi_{j,\chi}-A\mu_0+\inf_{\hat{U}} \varphi_{\hat{U}}-1\\
&& \geq 3 \tilde{\epsilon} \log r-C_3,
\end{eqnarray*}
where the constant $C_3$ depends only on $\phi_{j,\omega_0}$, $\phi_{j,\chi}$, $\phi_{j,\xi}$, $\phi$ and $\varphi_{\hat{U}}$. Now let $i \in \cJ$ be any index such that $z \in B_{i,3R}$. From our choice of the finite covering $\{B_{i',3R}\}_{i' \in \cJ}$, one can easily see that $B_{j,\frac{r}{2}}(z) \subset B_{i,r}(z)$ for all $r \in (0,\min \{r_0,R/2\})$. Thus we have
\begin{eqnarray*}
\psi_{j,\frac{r}{2}}(z) &\leq& \sup_{B_{j,\frac{r}{2}}(z)} \varphi+ \sup_{B_{j,\frac{r}{2}}(z)}(\phi_{j,\omega_0}-\d_0 \phi_{j,\chi}+A \phi_{j,\chi})\\
&\leq& \sup_{B_{i,r}(z)} \varphi+4KR^2\\
&\leq& \psi_{i,r}(z)+\sup_{B_{i,r}(z)} \big(-(\phi_{i,\omega_0}-\d_0 \phi_{i,\chi}+A \phi_{i,\chi}) \big)+4KR^2\\
&\leq& \psi_{i,r}(z)+8KR^2,
\end{eqnarray*}
and hence
\[
\psi_{i,r}(z) \geq 3 \tilde{\epsilon} \log \bigg( \frac{r}{2} \bigg)-8KR^2-C_3.
\]
This shows that
\[
4\tilde{\epsilon} \bigg(\log \bigg( \frac{3}{4} R \bigg)-\log r \bigg)-\big(\psi_{i,\frac{3}{4} R}(z)-\psi_{i,r}(z) \big) \geq -\tilde{\epsilon} \log r-C_4
\]
for some constant $C_4>0$ depending only on $K$, $R$, $C_3$ and $\psi_{i, \frac{3}{4}R}$. So if we take $r_1<e^{-C_4/\tilde{\epsilon}}$, we have
\[
\nu_{i, \psi_i}(z,r)<4 \tilde{\epsilon}
\]
for all $r \in (0,r_1)$.}

\noindent
(3) We assume that a point $z \in \hat{Y}$ satisfies
\[
z \in (B_{j,3R} \backslash B_{j,2R}) \cap B_{i,R}, \quad \max_{\{j' \in \cJ|z \in B_{j',3R}\}} \nu_{j',\psi_{j'}}(z,r) \leq 4 \tilde{\epsilon},
\]
where we note that the set $\{i \in \cJ|z \in B_{i,R}\}$ is not empty since we assume that $\{B_{i,R}\}$ is also a covering of $\hat{Y}$. Then for any $r<\min(r_0,R/2)$, we compute
\begin{eqnarray*}
&& \tilde{\varphi}_{j,r}(z)\\[2mm]
&&=\psi_j^{(r)}(z)-A \phi_{j,\chi}^{(r)}(z)-\phi_{j,\omega_0}(z)+\d_0 \phi_{j,\chi}(z)-\d_0^3 \phi_{j,\xi}(z)+\d_0^2 \phi (z)\\[2mm]
&& \leq \psi_{j,\frac{r}{2}}(z)+(\log 2) \nu_{j,\psi_j}(z,r)-A \phi_{j,\chi}^{(r)}(z)-\phi_{j,\omega_0}(z)+\d_0 \phi_{j,\chi}(z)-\d_0^3 \phi_{j,\xi}(z)+\d_0^2 \phi(z)\\[2mm]
&& \leq \psi_{j,\frac{r}{2}}(z)+(\log 2) \nu_{j,\psi_j}(z,r)-A \phi_{j,\chi}(z)+A\mu_0-\phi_{j,\omega_0}(z)+\d_0 \phi_{j,\chi}(z)-\d_0^3 \phi_{j,\xi}(z)+\d_0^2 \phi(z)\\[2mm]
&& \leq \sup_{B_{j,\frac{r}{2}}(z)} \varphi+\sup_{B_{j,\frac{r}{2}}(z)}\big( \phi_{j,\omega_0}-\d_0 \phi_{j,\chi}+A \phi_{j,\chi}-(\phi_{j,\omega_0}-\d_0 \phi_{j,\chi}+A \phi_{j,\chi})(z) \big)+4 \tilde{\epsilon}+A \mu_0\\
&&-\d_0^3 \phi_{i,\xi}(z)-\d_0^3(\phi_{j,\xi}(z)-\phi_{i,\xi}(z))+\d_0^2 \phi(z).
\end{eqnarray*}
Since $B_{j,\frac{r}{2}}(z) \subset B_{i,r}(z)$ for all $r<\min(r_0,R/2)$, we have
\begin{eqnarray*}
&& \tilde{\varphi}_{j,r}(z)\\
&& \leq \sup_{B_{i,r}(z)} \varphi+KRr+4 \tilde{\epsilon}+A\mu_0-\d_0^3 \phi_{i,\xi}(z)-\d_0^3 \bigg( \bigg(\frac{7}{4}R \bigg)^2-\bigg( \frac{5}{4}R \bigg)^2 \bigg)+\d_0^2 \phi(z)\\
&& \leq \psi_{i,r}(z)+\sup_{B_{i,r}(z)} \big(-(\phi_{i,\omega_0}-\d_0 \phi_{i,\chi}+A\phi_{i,\chi}) \big)+KRr+4 \tilde{\epsilon}+A\mu_0-\d_0^3 \phi_{i,\xi}(z)-\frac{3}{2} \d_0^3R^2+\d_0^2 \phi(z)\\
&& \leq \psi_i^{(r)}(z)+\sup_{B_{i,r}(z)} \big(-(\phi_{i,\omega_0}-\d_0 \phi_{i,\chi}+A\phi_{i,\chi}) \big)+(4+4\eta)\tilde{\epsilon}+KRr+A\mu_0-\d_0^3 \phi_{i,\xi}(z)\\
&& -\frac{3}{2} \d_0^3R^2+\d_0^2 \phi(z)\\[2mm]
&& \leq \tilde{\varphi}_{i,r}(z)+\sup_{B_{i,r}(z)} \big((\phi_{i,\omega_0}-\d_0 \phi_{i,\chi}+A\phi_{i,\chi})(z)-(\phi_{i,\omega_0}-\d_0 \phi_{i,\chi}+A\phi_{i,\chi}) \big)+(4+4\eta)\tilde{\epsilon}\\
&& +KRr+2A\mu_0-\frac{3}{2} \d_0^3R^2\\
&& \leq \tilde{\varphi}_{i,r}(z)+(4+4\eta)\tilde{\epsilon}+2KRr+2A\mu_0-\frac{3}{2} \d_0^3R^2\\[2mm]
&& \leq \tilde{\varphi}_{i,r}(z)-2 \tilde{\epsilon}
\end{eqnarray*}
from the choice of $\tilde{\epsilon}$ \eqref{choice of tilde epsilon} and $\mu_0$ \eqref{choice of mu zero} as long as $r \leq \frac{\d_0^3R}{8K}$.
\end{proof}
Thus we may take the regularized maximum of $(B_{j,3R},\tilde{\varphi}_{j,r})$ and $(\hat{U}, \varphi_{\hat{U}})$. By Lemma \ref{gluing condition}, we know that the resulting function $\tilde{\varphi}$ is smooth. If we set
\[
\Omega':=\omega_0+\dd \tilde{\varphi},
\]
then by \eqref{local subsolution away from e zero}, \eqref{local subsolution near e zero}, we know that $\Omega'$ is smooth on $\hat{Y}$ and satisfies
\[
Q_\chi (\Omega') \leq \theta_0-\epsilon_0
\]
on $\hat{Y} \backslash E_0$. Thus there exists a sufficiently small constant $s>0$ such that
\[
Q_\chi (\Omega'-s \chi)<\theta_0
\]
on $\hat{Y} \backslash E_0$ by using the fact that $\arccot(\lambda-s) \leq \arccot(\lambda)+s$ for all $\lambda \in \R$ and $s>0$. We set
\[
\tilde{\Omega}:=\Omega'+s\dd \phi=\Omega'-s\chi+s \xi \geq \Omega'-s\chi.
\]
Then the monotonicity of $Q_\chi$ shows that
\[
Q_\chi(\tilde{\Omega})<\theta_0
\]
on $\hat{Y} \backslash E_0$. Moreover, if we write $\tilde{\Omega}$ as
\[
\tilde{\Omega}=\omega_0+\dd \varphi_{\hat{Y}}, \quad \varphi_{\hat{Y}}:=\tilde{\varphi}+s\phi,
\]
then $\varphi_{\hat{Y}} \to -\infty$ uniformly at $E_0$.
\begin{proof}[Proof of Theorem \ref{gluing of local subsolutions}]
Theorem \ref{gluing of local subsolutions} follows immediately from the above observations and the fact that $Y \backslash Y_{\sing}=\Phi \big(\hat{Y} \backslash E_0 \big)$.
\end{proof}

\section{Completion of the proof of Theorem \ref{NM criterion A}} \label{completion of the proof}
First, we will give the proof of Theorem \ref{NM criterion B}.
\begin{proof}[Proof of Theorem \ref{NM criterion B}]
The argument is similar as the $J$-equation case \cite[Theorem 3.1]{Son20}. By the induction assumption, there exists an open neighbourhood $U$ of $Y_{\sing}$ in $X$ and $\varphi_U \in C^\infty(U)$ such that
\[
\cQ_{\chi,n,n}(\omega_U)<\theta_0, \quad \omega_U:=\omega_0+\dd \varphi_U.
\]
We take $\varphi_Y$ as in Theorem \ref{gluing of local subsolutions}, and open neighbourhoods $W \Subset V \Subset U$ of $Y_{\sing}$ in $X$ such that
\begin{enumerate}\setlength{\itemsep}{1mm}
\item both $Y \backslash W$ and $Y \backslash V$ are a union of finitely many open smooth analytic subvarieties of $Y$;
\item $\varphi_Y<\varphi_U-2$ in $Y \cap W$;
\item $\varphi_Y>\varphi_U+2$ in $Y \cap (U \backslash V)$
\end{enumerate}
by subtracting a large constant from $\varphi_U$ if necessary since $\varphi_Y$ tends to $-\infty$ uniformly along $Y_{\sing}$. We take a smooth extension $\varphi'_Y$ of $\varphi_Y$ in $Y \backslash W$ and set
\[
\varphi''_Y:=\varphi'_Y+A d^2
\]
for a constant $A>1$, where $d$ denotes the distance function to $Y \backslash W$ with respect to any fixed K\"ahler metric on $X$. By taking a sufficiently large $A$, we observe that
\[
\cQ_{\chi,n,n}(\omega''_Y)<\theta_0, \quad \omega''_Y:=\omega_0+\dd \varphi''_Y
\]
in an open neighbourhood of $Y \backslash W$ in $X$. Since $\varphi''_Y$ is continuous, there exists a sufficiently small open neighbourhood $\tilde{U}$ of $Y \backslash W$ in $X$ such that
\[
\varphi''_Y<\varphi_U-1
\]
in $\tilde{U} \cap W$ and
\[
\varphi''_Y>\varphi_U+1
\]
in $\tilde{U} \backslash V$. So if we let $\varphi$ be the regularized maximum of $(\tilde{U}, \varphi''_Y)$ and $(V, \varphi_U)$, then we have
\[
\cQ_{\chi,n,n}(\omega)<\theta_0, \quad \omega:=\omega_0+\dd \varphi
\]
on a neighbourhood of $Y$ in $X$ as desired.
\end{proof}
Now we will prove Theorem \ref{NM criterion A}. Again we stress that once we obtain Theorem \ref{NM criterion B}, we can prove Theorem \ref{NM criterion A} by using exactly the same argument as \cite[Section 5]{Che21}. However, we will give the outline of the proof for reader's convenience.
\begin{proof}[Outline of the proof of Theorem \ref{NM criterion A}]
Assume that the triple $(X,\a,\b)$ is stable along a test family $\omega_{t,0} \in \a_t$ ($t \in [0,\infty)$).  By Theorem \ref{NM criterion B}, the definition (B) of the test family and monotonicity of $Q_\chi$, we know that
\begin{equation} \label{time dependent subsolution}
\Gamma_{\chi,\a_t,\theta_0,\theta_0}(Y,X) \neq \emptyset
\end{equation}
for all proper analytic subvariety $Y \subset X$ and $t \in [0,\infty)$. Let us consider the following twisted dHYM equation for $\omega_t \in \a_t$
\begin{equation} \label{twisted dHYM equation on X}
\Re(\omega_t+\i \chi)^n-\cot(\theta_0) \Im(\omega_t+\i \chi)^n-c_t \chi^n=0,
\end{equation}
where the normalized constant $c_t$ is non-negative for all $t \in [0,\infty)$ and $c_0=0$ by the stability assumption and choice of $\theta_0$. We take a constant $\Theta_0 \in (\theta_0,\pi)$ (depending only on $n$ and $\theta_0$) sufficiently close to $\theta_0$ so that
\[
\Theta_0-\theta_0<\frac{\pi-\Theta_0}{n},
\]
and define
\[
\cT:=\{t \in [0,\infty)|\text{\eqref{twisted dHYM equation on X} has a solution $\omega_t \in \Gamma_{\chi,\a_t,\theta_0, \Theta_0}(X)$ for $t$} \}.
\]
By (C) of the definition of the test family, monotonicity of $P_\chi$, $Q_\chi$ and Proposition \ref{continuity method}, we know that there exists a large enough $T \geq 0$ such that $[T,\infty) \subset \cT$. Also $\cT$ is open by Proposition \ref{continuity method}. To show the closeness, we have to prove that $t_0 \in \cT$ for $t_0:=\inf \cT$. Now we apply the same argument in Section \ref{the twisted dHYM equation on the product space}, \ref{the mass concentration and local smoothing} (essentially based on \cite[Section 5]{Che21}) with suitable small changes: first, we take a constant $K>0$ (depending only on $n$ and $\theta_0$) so that
\begin{equation} \label{choice of K}
\cot \big( \frac{\pi-\Theta_0}{n} \big)<K<\cot(\Theta_0-\theta_0),
\end{equation}
and set
\[
0<\tilde{\theta}_0:=\theta_0+n \arccot(K)<\tilde{\Theta}_0:=\Theta_0+n \arccot(K)<\pi.
\]
Set $\cX:=X \times X$, and define the quantities $\chi_\cX$, $\chi_{\cX,s}$, $F_{\cX,s}$ corresponding to \eqref{definition of chi y}, \eqref{definition of chi ys}, \eqref{definition of F ys} respectively (but we set $F_{h_{E_0}}=0$ in this case). We consider the twisted dHYM equation for $\omega_{\cX,s} \in \pi_1^\ast \a_t+K \pi_2^\ast \b$ ($t \in (t_0, \infty)$)
\begin{equation} \label{twisted dHYM equation on the product space of X}
\Re(\omega_{\cX,s}+\i \chi_\cX)^{2n}-\cot(\tilde{\theta}_0) \Im (\omega_{\cX,s}+\i \chi_\cX)^{2n}-F_{\cX,s} \chi_\cX^{2n}=0.
\end{equation}
Then by using the condition \eqref{choice of K}, one can easily see that $\omega_\cX^\dag:=\pi_1^\ast \omega_t+K \pi_2^\ast \chi$ satisfies $P_{\chi_\cX}(\omega_\cX^\dag)<\tilde{\theta}_0$ and $Q_{\chi_\cX}(\omega_\cX^\dag)<\tilde{\Theta}_0$, where $\omega_t$ denotes the solution to \eqref{twisted dHYM equation on X} for $t$. Also by choosing a sufficiently small constant $\ell>0$ in the definition of $\chi_{\cX,s}$, one can observe that $\inf F_{\cX,s}>-c$ where the constant $c>0$ is determined in Proposition \ref{continuity method}. So we obtain a solution $\omega_{\cX,s} \in \Gamma_{\chi_\cX,\pi_1^\ast \a_t+K \pi_2^\ast \b,\tilde{\theta}_0,\tilde{\Theta}_0}(\cX)$ for all $t \in (t_0, \infty)$. Then performing the truncation and using the mass concentration argument, one can proceed the computations for $P_\chi$, $Q_\chi$ almost in parallel, by using the monotonicity and concavity of these functions. As a consequence, we obtain a constant $\d_0>0$ as in Lemma \ref{mass concentration} and an analogous result of Theorem \ref{local subsolution}, where the condition \eqref{local smoothing Q} for a current $\Omega \in \a_{t_0}-\d_0 \b$ should be replaced by
\[
P_\chi(\Omega^{(r)}) \leq \theta_0-\epsilon_0, \quad Q_\chi(\Omega^{(r)}) \leq \Theta_0-\epsilon_0.
\]
Finally, we consider the gluing argument as in Section \ref{gluing construction} (essentially based on \cite[Section 5]{Che21}). If $\Omega+A\chi$ has no positive Lelong number, we are done. If not, we apply \eqref{time dependent subsolution} to the $\tilde{\epsilon}$-sublevel set $E_{\tilde{\epsilon}}$ of Lelong number of $\Omega+A\chi$ for some suitable small constant $\tilde{\epsilon}>0$ and then glue them together. As a result, we obtain a subsolution $\omega_X \in \Gamma_{\chi,\a_{t_0},\theta_0,\Theta_0}(X)$. By Proposition \ref{continuity method}, this implies that $t_0 \in \cT$. Thus we have a solution of \eqref{twisted dHYM equation on X} at $t=0$. This completes the proof.
\end{proof}


\end{document}